%% file: main.tex
\pdfoutput=1
\documentclass[toc]{mathprint} %
\usepackage{rutarmacros}
\usepackage{macros}

\title[Sets of continued fractions]{Multi-scale properties of continued fraction sets}

\author
  {Alex Rutar}
  {Department of Mathematics and Statistics, University of Jyväskylä, P.O.\ Box 35 (MaD), FI-40014 University of Jyväskylä, Finland}
  {alex@rutar.org}

\begin{document}
\begin{abstract}
    We survey the dimension theory of sets of real numbers with regular continued fraction expansion restricted to a non-empty and possibly infinite subset $\D\subset\N$.
\end{abstract}

\section{Continued fractions with restricted digits}

\subsection{Regular continued fractions}
Let $x\in\R$.
A \emph{(regular) continued fraction expansion} of $x$ is representation of the form
\begin{equation*}
    x = a_0 + \cfrac{1}{a_1 + \cfrac{1}{a_2 + \cfrac{1}{a_3 + \cfrac{1}{\ddots}}}}
\end{equation*}
where $a_n\in\Z$ and $a_n \geq 1$ for $n \neq 0$.
We will write $x = [a_0; a_1, a_2, a_3, \ldots]$.
It is easy to check that the continued fraction expansion is infinite if and only if $x$ is irrational, in which case the continued fraction expansion is also unique.
If $x$ is rational then there are two finite continued fraction expansions.

Henceforth, we only consider numbers $x\in\Ir\coloneqq(0,1)\setminus\Q$ and write $x = [a_1,a_2,a_3,\ldots] = [0; a_1, a_2, a_3, \ldots]$.
We are primarily interested in the following subsets of $\Ir$.
For an arbitrary non-empty (and possibly infinite) subset $\D\subset\N$, we write
\begin{equation*}
    \Lambda_{\D} = \{x = [a_1,a_2,a_3,\ldots]: a_n\in\D\text{ for all }n \in\N\}.
\end{equation*}
This is the set of numbers for which the continued fraction expansion only contains digits in $\D$.
A classical example is the case $\D = \{1,\ldots,\tau\}$ for some $\tau\in\N$, in which case $\Lambda_{\D}$ is the set of $\tau$-badly approximable numbers.

The continued fraction expansion can be obtained in a natural way using the \emph{Gauss map}.
The Gauss map is the transformation $\Gamma\colon\Ir\to\Ir$ defined by the rule
\begin{equation}\label{e:Gamma-def}
    \Gamma(x) \coloneqq \left\{\frac{1}{x}\right\} = \frac{1}{x} - \left\lfloor\frac{1}{x}\right\rfloor.
\end{equation}
If $x = [a_1,a_2,a_3,a_4,\ldots]$, then $\Gamma(x) = [a_2,a_3,a_4,\ldots]$, so $\Gamma$ acts like a left shift on the continued fraction expansion of $x$.
In particular, the sets $\Lambda_{\D}$ are \emph{forward invariant} for the Gauss map: $\Gamma(\Lambda_{\D}) = \Lambda_{\D}$.
These are certainly not the only invariant sets!
However, they are perhaps the simplest family of invariant sets.

\subsection{Infinitely-generated self-conformal sets}
We now see the representation of the sets $\Lambda_{\D}$ which will be of primary use in this survey: as the attractor of an infinitely-generated self-conformal IFS.

The Gauss map is far from invertible, in that $\Gamma^{-1}(y)$ is a countably infinite set for all $y$.
On the other hand, for $b\in\N$, if we restrict $\Gamma$ to the interval $J_b \coloneqq ((b+1)^{-1}, b^{-1})\cap\Ir$, then $\Gamma$ is invertible and $\Gamma(J_b) = \Ir$.
We call the restriction $\Gamma|_{J_b}$ a \emph{branch} of $\Gamma$.

We can invert each branch to obtain a map $\phi_b = \Gamma_{J_b}^{-1}\colon\Ir \to J_b \subset \Ir$.
Concretely, $\phi_b(x) = (b + x)^{-1}$, which we can think of as a map from $[0,1]$ to $[0,1]$.
Then the set $\Lambda_{\D}$ can equivalently be obtained as the maximal invariant set for the family of maps $\{\phi_b\}_{b\in\D}$: that is, the union over all sets $E\subset [0,1]$ for which
\begin{equation*}
    E = \bigcup_{b\in\D}\phi_b(E).
\end{equation*}
This allows us to place the study of the sets $\Lambda_{\D}$ in a somewhat more general context.
\begin{definition}
    Suppose $(X, d)$ is a compact metric space and $\mathcal{I}$ is a countable index set.
    A countable family of maps $\{f_i\}_{i\in\mathcal{I}}$ from $X$ to itself is \emph{(uniformly) contracting} if there is a constant $\xi<1$ such that each $f_n$ is a Lipschitz contraction with constant $\xi$:
    \begin{equation*}
        d(f_n(x), f_n(y)) \leq \xi d(x, y)\qquad\text{for all}\qquad x,y\in X.
    \end{equation*}
\end{definition}
Similarly to the case for the sets $\Lambda_{\D}$, given a contracting family of maps $\{f_i\}_{i\in\mathcal{I}}$, there is a unique maximal invariant set
\begin{equation*}
    \Lambda = \bigcup_{i\in\mathcal{I}}f_i(\Lambda).
\end{equation*}
We call the family $\{f_i\}_{i\in\mathcal{I}}$ an \emph{iterated function system} (or \emph{IFS} for short) and the invariant set $\Lambda$ the \emph{attractor}.

An IFS $\{f_i\}_{i\in\mathcal{I}}$ comes equipped with a surjective \emph{coding map} $\pi\colon\mathcal{I}^{\N} \to \Lambda$ defined by the rule
\begin{equation*}
    \{\pi((i_n)_{n=1}^\infty)\} = \lim_{n\to\infty}f_{i_1}\circ\cdots\circ f_{i_n}(X).
\end{equation*}
The contracting property and compactness of $X$ ensures that the function $\pi$ is well-defined, and it is also easy to see that $\pi$ is continuous when $\mathcal{I}^{\N}$ is equipped with the product topology.
In particular, we alternatively note that $\Lambda = \pi(\mathcal{I}^{\N})$.
For example, in the special case of the IFS $\{\phi_b\}_{b\in\D}$, $\pi$ is the map which sends a sequence $(a_n)_{n=1}^\infty$ to the corresponding continued fraction $[a_1,a_2,a_3,\ldots]$, and $\pi\circ\sigma = \Gamma\circ \pi$ where $\sigma\colon\D^{\N}\to\D^{\N}$ is the left shift map.

Studying general attractors of infinite iterated function systems is rather intractable.
Later on, we will make more assumptions on the maps $\{f_i\}_{i\in\mathcal{I}}$.
These assumptions will be satisfied by the IFS $\{\phi_b\}_{b\in\D}$.
\begin{remark}
    The maps $\phi_b$ are not literally contractions since $\phi_b'(1) = 1$.
    However, the maps $\phi_b$ for $b \geq 2$ are uniformly contracting, and we can replace the map $\phi_1$ with the set of maps $\{\phi_1\circ \phi_b: b\in\D\}$ (which is a uniformly contracting family), and the resulting invariant set is unchanged by this operation.
    We will ignore this problem and just speak of the maps $\{\phi_b\}_{b\in\D}$ as uniformly contracting even when $1\in\D$.
\end{remark}
\begin{remark}
    Given a finite set of contracting Lipschitz maps $\{f_1,\ldots,f_m\}$, the transformation $K\mapsto f_1(K) \cup\cdots\cup f_m(K)$ maps compact sets to compact sets, and is in fact a Lipschitz contraction on the space of compact subsets of $XR$ with the Hausdorff metric, in which case the maximal invariant set $\Lambda = f_1(\Lambda) \cup\cdots\cup f_n(\Lambda)$ is the unique non-empty compact set with this property.

    However, if there are infinitely many maps, then $\Lambda$ need not be compact.
    In the general countably infinite case, the set $\Lambda$ is necessarily analytic, being a continuous image of the Polish space $\mathcal{I}^{\N}$ under the coding map $\pi$.
    However, as far as I am aware, it is not known if $\Lambda$ is necessarily Borel.
    In the special case that the images $f_n(X)$ do not overlap too much (say, each $x\in X$ is only contained in finitely many images $f_n(X)$), then $\Lambda$ is a $F_{\sigma\delta}$ subset of $X$.
    This is the case for the sets $\Lambda_{\D}$.
\end{remark}
It is convenient, when working with repeated compositions, to have some shorthand.
We denote by $\mathcal{I}^* = \bigcup_{n=0}^\infty\mathcal{I}^n$ all of the finite words of length $n$; we denote by $\varnothing$ the unique word of length $0$.
Elements of $\mathcal{I}^*$ will be denoted in type-writer font, like $\mtt{i}$ or $\mtt{j}$.
The space $\mathcal{I}^*$ is equipped with the operation of concatenation.
For $\mtt{i} = (i_1,\ldots,i_n)\in\mathcal{I}^*$, we write
\begin{equation*}
    f_{\mtt{i}} = f_{i_1}\circ\cdots\circ f_{i_n}.
\end{equation*}
We say that $\mtt{i}$ is a prefix of $\mtt{j}$, and write $\mtt{i}\prec\mtt{j}$, if there is a $\mtt{k}\in\mathcal{I}^*$ such that $\mtt{i}\mtt{k} = \mtt{j}$,
If $\mtt{i}=(i_1,\ldots,i_n) \neq \varnothing$, we also write $\mtt{i}^{-} = (i_1,\ldots,i_{n-1})$; in other words, $\mtt{i}^{-}$ is the unique prefix of $\mtt{i}$ of length $n-1$.

\subsection{Hausdorff and box dimensions}
We now briefly introduce the notions of fractal dimension which will be the most important for us in this note.
The most fundamental notion is the \emph{Hausdorff dimension} which may be defined as follows.
For $0 <\delta \leq \infty$ and $s \geq 0$, we define the subadditive set function
\begin{equation*}
    \mathcal{H}^s_\delta(E) = \inf\left\{\sum_n (\diam A_n)^s: E\subset\bigcup_{n=1}^\infty A_n, \diam A_n \leq \delta\right\}.
\end{equation*}
The \emph{Hausdorff measure} $\mathcal{H}^s(E)\coloneqq\lim_{\delta\to 0}\mathcal{H}^s_\delta(E)$ is a metric outer measure and defines a Borel measure on $\R^d$.
The quantity $\mathcal{H}^s_\infty(E)$ is called \emph{Hausdorff content}.
Hausdorff measure (or equivalently Hausdorff content) can be used to define \emph{Hausdorff dimension}:
\begin{equation*}
    \dimH E = \sup\{s: \mathcal{H}^s_\infty(E) > 0\} = \inf\{s:\mathcal{H}^s(E) < \infty\}
\end{equation*}

Hausdorff dimension is quite measure-theoretically robust.
However, there is also a much conceptually simpler notion of dimension called the \emph{box} or \emph{Minkowski} dimension.
For a bounded subset $E\subset\R^d$ and $r>0$, we let $N_r(E)$ denote the least number of balls of radius $r$ required to cover $E$.
Then the \emph{lower} and \emph{upper box dimensions} are given respectively by
\begin{equation*}
    \dimlB E = \liminf_{r\to 0}\frac{\log N_r(E)}{\log(1/r)}\qquad\text{and}\qquad\dimuB E = \limsup_{r\to 0}\frac{\log N_r(E)}{\log(1/r)}.
\end{equation*}
When the limits coincide, we call the shared value the \emph{box dimension} and denote it by $\dimB E$.
\begin{remark}
    The terminology \emph{Minkowski dimension} originates in an alternative formulation of the box dimension due to Minkowski.
    Given a set $E$, let $E^{(r)}$ denote the open $r$-neighbourhood of $E$ and let $m$ denote Lebesgue measure in $\R^d$.
    Then
    \begin{equation*}
        \vol(E^{(r)}) \approx r^d N_r(E)
    \end{equation*}
    so
    \begin{equation*}
        \dimuB E = \limsup_{r\to 0}\frac{\log (r^{-d} \vol(E^{(r)}))}{\log(1/r)} =d-\liminf_{r\to 0}\frac{\log \vol(E^{(r)})}{\log r}.
    \end{equation*}
    Of course, an analogous formula holds for $\dimlB E$.
    The definition in terms of covers is typically more convenient since it is independent on the embedding in ambient space.
\end{remark}
It is easy to check that
\begin{equation*}
    \dimH E \leq \dimlB E \leq \dimuB E
\end{equation*}
and, for general subsets of $\R^d$, the dimensions can be distinct.
Essentially, the main difference between the Hausdorff and lower box dimensions is that the Hausdorff dimension permits covers using balls of any radius, whereas the lower box dimension only permits covers of a fixed radius $r$.

\subsection{Overview}
We will study the Hausdorff dimension and the lower and upper box dimensions of sets of continued fractions with restricted digits.
More generally, we will work with a general IFS $\{f_i\}_{i\in\mathcal{I}}$ with attractor $\Lambda$, for a finite or countably infinite index set $\mathcal{I}$, acting on a compact subset $X\subset\R^d$.

Of particular interest is the family of maps $\{\phi_b\}_{b\in\D}$ of inverse branches of the Gauss map, and the corresponding restricted digit sets $\Lambda_{\D}$.
The non-linearity of the maps $\phi_b$ results in a decent amount of technical difficulty, which obscures the overall picture of the proof.
Therefore, we will first handle the special case that the maps $f_n$ are in \emph{similarities}, and only later on \cref{ss:conf} explain what needs to be done to handle general conformal maps.
\begin{itemize}[nl]
    \item In \cref{s:h-ub} we provide an introduction to infinitely-generated self-similar sets, and prove results concerning the Hausdorff and upper box dimensions.
        These results are originally due to Mauldin--Williams \cite{zbl:0625.54047} and Mauldin--Urbański \cite{zbl:0852.28005,zbl:0940.28009}.
    \item In \cref{s:scaling} we turn our attention to the lower box dimension, still in the self-similar case.
        These results are originally from \cite{arxiv:2406.12821}.
        It turns out that there is no elegant formula for the lower box dimension, so we will study the more general question of asymptotics of the covering numbers $r\mapsto N_r(\Lambda)$.
        Then, we will use this to say something about the lower box dimension.
        For example, we can obtain a simple classification of when the box dimension of $\Lambda$ exists.
    \item In \cref{s:final} we tie up some loose ends, albeit with less detail than the earlier results.
        Most notably, we discuss the sharpness of the bounds on the lower box dimension, and also provide a summary of the key differences in the self-conformal case and the main ideas of how to handle them.
        We also spend some time discussing other notions of dimension.
\end{itemize}

\subsection{A note on implicit constants}
Given functions $f, g\colon A\to [0,\infty)$ defined on some domain $A$, we write $f \lesssim g$ to indicate that there is an implicit constant $C>0$ such that $f(a) \leq C g(a)$ for all $a\in A$.
If $f\lesssim g$ and $g\lesssim f$, we also write $f\approx g$.

We will fix an IFS $\{f_i\}_{i\in\mathcal{I}}$ defined on a compact subset $X\subset\R^d$.
The implicit constants are permitted to depend on the ambient dimension $d$, the uniform upper bound $\xi$ on the contraction ratio, and $\diam X$.
Any other dependence will be indicated explicitly with a subscript, like $\lesssim_\varepsilon$.

\section{Infinitely-generated self-similar sets}\label{s:h-ub}
\subsection{Self-similar iterated function systems}
Let us begin by introducing the key concepts concerning self-similar iterated function systems.
\begin{definition}
    Let $(X,d)$ be a metric space.
    A map $f\colon X\to X$ is a \emph{similarity} if there is a number $\lambda > 0$ so that for all $x,y\in X$,
    \begin{equation*}
        d(f(x),f(y)) = \lambda d(x,y).
    \end{equation*}
    We call the number $\lambda$ the \emph{similarity ratio}.
\end{definition}
In $\R^d$, similarity maps can always be written in the form $f(x) = \lambda O x + t$ where $O$ is an orthogonal matrix and $t\in\R^d$ is a translation.

A \emph{self-similar IFS} is an IFS $\{f_i\}_{i\in\mathcal{I}}$ where the maps $f_n$ are contracting similarities.
Given $\mtt{i}\in\mathcal{I}^*$, we let $\rho(\mtt{i}) \in (0,1]$ denote the similarity ratio of $f_{\mtt{i}}$.
Observe that $\rho$ is multiplicative: $\rho(\mtt{i}\mtt{j}) = \rho(\mtt{i})\rho(\mtt{j})$.
The quantity $\rho$ is important since $f_{\mtt{i}}(X)$ is essentially a copy of $X$ scaled down by a factor of $\rho(\mtt{i})$.

Assuming that the maps $f_{\mtt{i}}$ are similarity maps is very useful, but without more assumptions it is still rather difficult to understand the dimension theory of the attractor $\Lambda$.
Essentially, the difficulty is that the sets $f_{\mtt{i}}(X)$ may overlap substantially, in which case the geometry of $X$ depends not only on the similarity ratios of the maps $f_{\mtt{i}}$ but also on how the images $f_{\mtt{i}}(X)$ are arranged in space.
Even in the case that $\mathcal{I}$ is finite, this is a very difficult problem and dimension of the attractor is not known in general.
Some of the most general results concerning the Hausdorff dimension in the overlapping case can be found in \cite{zbl:1337.28015,zbl:1426.28024,arxiv:2503.21923,zbmath:7835917,arxiv:2501.17795,zbl:1527.28009}.

In the special case of continued fractions, however, the functions $\phi_b$ were obtained as the inverse branches of the Gauss map, so the sets $\phi_b(\Ir)$ are disjoint for distinct values of $n$.
The following assumption is a close sibling to such a disjointness assumption.

We say that a subset $\mathcal{F}\subset\mathcal{I}^*$ is \emph{mutually incomparable} if $\mtt{i}$ is not a prefix of $\mtt{j}$ for all $\mtt{i},\mtt{j}\in\mathcal{F}$ with $\mtt{i}\neq\mtt{j}$.
\begin{definition}\label{d:bnc}
    We say that the IFS $\{f_i\}_{i\in\mathcal{I}}$ satisfies the \emph{bounded neighbourhood condition} if there exists $M\in\N$ so that for all mutually incomparable $\mathcal{F}\subset\mathcal{I}^*$, for all $x\in X$, and for all $r\in(0,1)$,
    \begin{equation*}
        \#\{\mtt{i}\in\mathcal{F}:\rho(\mtt{i})>r\text{ and }f_{\mtt{i}}(X)\cap B(x,r)\neq\varnothing\}\leq M.
    \end{equation*}
\end{definition}
\begin{example}
    A simple example of an IFS of similarities can be obtained from the \emph{Lüroth expansion} on $[0,1]$.
    For each $n\in\N$, let $g_n\colon[0,1]\to [0,1]$ denote the unique orientation-preserving affine map which maps $[0,1]$ to the interval $[1/(n+1), 1/n]$.
    More precisely,
    \begin{equation*}
        g_n(x) = \frac{x}{n(n+1)} + \frac{1}{n+1}.
    \end{equation*}
    The IFS $\{g_n\}_{n=1}$ is the set of inverse branches of the transformation
    \begin{equation*}
        T(x) = \left\lfloor\frac{1}{x}\right\rfloor\left(\left\lfloor\frac{1}{x}\right\rfloor+1\right) x - \left\lfloor\frac{1}{x}\right\rfloor
    \end{equation*}
    and produces expansions of the form
    \begin{equation*}
        x = \frac{1}{a_1} + \frac{1}{a_1(a_1-1)a_2}+\cdots+\frac1{a_n\cdot\prod_{j=1}^{n-1}a_j(a_j-1)} + \cdots.
    \end{equation*}
    This is a linearized version of the Gauss map, with the branches flipped.
    
    This IFS satisfies the bounded neighbourhood condition with constant $3$ since, for a mutually incomparable family $\mathcal{F}$, the sets $f_{\mtt{i}}((0,1))$ for $\mtt{i}\in\mathcal{F}$ are disjoint open intervals of width $\rho(\mtt{i})$.
\end{example}

\subsection{Hausdorff dimension of the attractor}
When $\mathcal{I}$ is a finite index set and $\{f_i\}_{i\in\mathcal{I}}$ is an IFS of similarity maps satisfying the bounded neighbourhood condition, it is well-known that the attractor $\Lambda$ satisfies
\begin{equation*}
    \dimH\Lambda = s \qquad\text{where}\qquad \sum_{i\in\mathcal{I}}\rho(i)^s = 1.
\end{equation*}
The upper bound is immediate, since the family of sets $\{f_{\mtt{i}}(X)\}_{\mtt{i}\in\mathcal{I}^{\N}}$ is a cover for $\Lambda$ with
\begin{align*}
    \sum_{\mtt{i}\in\mathcal{I}^{n}}(\diam f_{\mtt{i}}(X))^s
    &= (\diam X)^s \sum_{\mtt{i}\in\mathcal{I}^n}\rho(\mtt{i})^s\\
    &= (\diam X)^s \left(\sum_{i\in\mathcal{I}}\rho(i)^s\right)^n\\
    &= (\diam X)^s\\
    &< \infty.
\end{align*}
The lower bound requires slightly more work, but argument is essentially as follows.
Using the equation $\sum_{i\in\mathcal{I}}\rho(i)^s = 1$, define a product measure on the space $\bm{p}^{\N}$ on $\mathcal{I}^{\N}$ where $\bm{p}$ is the measure on $\mathcal{I}$ which assigns mass $\bm{p}(i) = \rho(i)^s$.
Let $\mu$ denote the pushforward of $\bm{p}^{\N}$ under the coding map $\pi$.
Then, using the bounded neighbourhood condition, check that $\mu$ is well-distributed, in that $\mu(B(x,r)) \leq C r^s$ for all $0<r\leq \diam X$ and $x\in\Lambda$.
Finally, we can use $\mu$ to obtain a lower bound on the $s$-cost of any cover of $\Lambda$: if $\Lambda \subset \bigcup_{n=1}^\infty A_n$, then
\begin{equation*}
    1 = \mu(\Lambda) \leq \sum_{n=1}^\infty \mu(B(x_n, \diam A_n)) \leq C \sum_{n=1}^\infty (\diam A_n)^s
\end{equation*}
where $x_n \in A_n$.
Since $A_n$ was an arbitrary cover, it follows that $\mathcal{H}^s(\Lambda) > 0$ and therefore $\dimH\Lambda \geq s$.

When $\mathcal{I}$ is infinite, we immediately run into an issue: the equation $\sum_{i\in\mathcal{I}}\rho(i)^s = 1$ may not have a solution!
Regardless, we can still say something useful.
Write
\begin{equation*}
    P(t) = \log \sum_{i\in\mathcal{I}}\rho(i)^t.
\end{equation*}
We call $P$ the \emph{pressure function} corresponding to the IFS.
Then, we set
\begin{equation}\label{e:h-def}
    h \coloneqq \inf\{t \geq 0: P(t) < 0\}.
\end{equation}
The exact same argument as in the finite case shows that $\dimH\Lambda \leq h$.
In fact, for all $\varepsilon > 0$,
\begin{equation}\label{e:total-bound}
    \sum_{\mtt{i}\in\mathcal{I}^*}\rho(\mtt{i})^{h + \varepsilon} = \sum_{n=0}^\infty \left(\sum_{i\in\mathcal{I}}\rho(i)^{h+\varepsilon}\right)^n \lesssim_\varepsilon 1.
\end{equation}

On the other hand, it need not hold that $P(h) = 0$.
It can happen that $P(t) = +\infty$ for small values of $t$, and the moment $P(t)$ is finite it already takes negative values bounded away from $0$.
To this extent, we write
\begin{equation*}
    \Omega = \{t\geq 0: P(t) <\infty\}\qquad\text{and}\qquad \eta = \inf \Omega.
\end{equation*}
We call $\eta$ the \emph{finiteness parameter}.
Clearly $0 \leq \eta \leq h$, and $\eta = 0$ if and only if $\mathcal{I}$ is finite.
Since $P$ is strictly decreasing, either $\Omega = (\eta,\infty)$ or $\Omega = [\eta,\infty)$.

Moreover, one can check that $P$ is a continuous and convex function on $\Omega$.
In particular, if for instance $\eta < h$, then necessarily $P(h) = 0$.

If $P(h) = 0$, one can (with rather more work) develop a similar theory as in the finite case $\mathcal{I}$ and define a well-distributed measure $\mu$ on the attractor $\Lambda$, and with that obtain a lower bound for the Hausdorff dimension of $\Lambda$.
This is done in detail in \cite{zbl:0852.28005}.
But if one is only interested in the Hausdorff dimension, it turns out that one only requires the results in the finite case to complete the proof.

Let $\Fin(\mathcal{I})$ denote the set of all finite subsets of $\mathcal{I}$.
Moreover, for $\varnothing\neq\mathcal{F}\subset\mathcal{I}$, we let $\Lambda_{\mathcal{F}}$ denote the attractor of the IFS $\{f_i\}_{i\in\mathcal{F}}$.
The following result is \cite[Theorem~3.15]{zbl:0852.28005}.
\begin{theorem}
    Let $\{f_i\}_{i\in\mathcal{I}}$ be an IFS of similarities satisfying the bounded neighbourhood condition, with pressure function $P(t)$ and attractor $\Lambda$.
    Then
    \begin{equation*}
        \dimH\Lambda = \inf\{t \geq 0: P(t) < 0\} = \sup\{\dimH \Lambda_{\mathcal{F}}: \mathcal{F}\in\Fin(\mathcal{I})\}.
    \end{equation*}
\end{theorem}
\begin{proof}
    We already saw that $\dimH\Lambda \leq h$.
    Since $\Lambda_{\mathcal{F}}\subset\Lambda$ for all $\mathcal{F}\in\Fin(\mathcal{I})$, it remains to show that
    \begin{equation*}
        h \leq \sup\{\dimH \Lambda_{\mathcal{F}}: \mathcal{F}\in\Fin(\mathcal{I})\}.
    \end{equation*}
    If $h = 0$ there is nothing to prove; otherwise, let $0 < s < h$ be arbitrary so that $P(s) > 0$ (note, perhaps, that $P(s) = +\infty$).
    Equivalently, $\sum_{i\in\mathcal{I}}\rho(i)^s > 1$.
    By the usual properties of summation, there is a finite subset $\mathcal{F}\subset\mathcal{I}$ such that
    \begin{equation*}
        \sum_{i\in\mathcal{F}}\rho(i)^s > 1.
    \end{equation*}
    But recalling the formula for $\dimH\Lambda_{\mathcal{F}}$, this means that $\dimH\Lambda_{\mathcal{F}} \geq s$.
    Since $s < h$ was arbitrary, the proof is complete.
\end{proof}
To summarize, the sets $f_{\mtt{i}}(X)$ provide natural families of covers, and the covers obtained in this way are optimal from the perspective of Hausdorff dimension.

\subsection{Upper box dimension of the attractor}
Now that we have a good picture of the Hausdorff dimension, let us turn our attention to the upper box dimension.
We continue to write $h = \inf\{t \geq 0: P(t) < 0\}$.

Unlike the Hausdorff dimension, the upper box dimension can be strictly greater than $0$ even for countable sets.
For example, it is a good exercise to verify that
\begin{equation*}
    \dimB \{n^{-1} : n\in\N\} = \frac{1}{2}.
\end{equation*}
The IFS of similarities $\{f_i\}_{i\in\mathcal{I}}$ contains a countable subset which may be large from the perspective of the upper box dimension.
For each map $f_n$, let $x_n$ denote the unique fixed point, and let $F = \{x_i : i\in\mathcal{I}\}$.
Certainly, $F\subset \Lambda$.
Moreover, it could happen that $\dimuB F > 0$, and this is an obstruction to the upper box dimension:
\begin{equation}\label{e:dimb-form}
    \dimuB \Lambda \geq \max\{h, \dimuB F\}.
\end{equation}
This seems like a rather trivial obstruction, but let us note two facts.
Firstly, one must also contend with images $f_{\mtt{i}}(F)$ for $\mtt{i}\in\mathcal{I}^*$, and this family of sets is rather more substantial.
Secondly, since $\dimuB\Lambda \cap V = \dimuB\Lambda$ for any open set $V$ intersecting $\Lambda$, by a Baire category argument, it in fact holds that $\dimuB\Lambda = \dimP\Lambda$, where $\dimP\Lambda$ is the packing dimension which (for our purposes) can be obtained as countably stabilized upper box dimension:
\begin{equation*}
    \dimP E = \inf\{\max_i \dimuB E_i: E\subset\bigcup_{i=1}^\infty E_i\}.
\end{equation*}
Since the packing dimension is countably stable, $\dimP F = 0$, but regardless $\dimP\Lambda \geq \dimuB F$.
\begin{remark}
    The precise choice of $F$ as the set of fixed points of the maps $f_i$ is convenient for the proofs, but it is not essential.
    For example, one could take any point $x_i \in f_i(X)$ for $i\in\mathcal{I}$ and let $F = \{x_i:i\in\mathcal{I}\}$.
    See \cref{ss:approx} for more detail.
\end{remark}

The following result was first proven in \cite{zbl:0940.28009}, with a somewhat different inductive proof.
We present an alternative proof since (at least in my opinion) the proof is more transparent, and moreover the technique generalizes more easily when we handle the lower box dimension.
\begin{theorem}\label{t:dimuB-ss}
    Let $\{f_i\}_{i\in\mathcal{I}}$ be an IFS of similarities with attractor $\Lambda$ and fixed point set $F$.
    Then
    \begin{equation*}
        \dimuB\Lambda \leq \max\{h, \dimuB F\}.
    \end{equation*}
    In particular, if $\dimH\Lambda \geq h$ (for instance, if the IFS satisfies the bounded neighbourhood condition), then equality holds.
\end{theorem}
\begin{proof}
    Let $\varepsilon > 0$ be arbitrary and let $r > 0$ be small.
    We will provide an upper bound for $N_r(\Lambda)$.

    First, iterate the self-similarity relation:
    \begin{equation*}
        N_r(\Lambda) = N_r\left(\bigcup_{i\in\mathcal{I}}f_i(\Lambda)\right).
    \end{equation*}
    At this point, we cannot swap the covering number $N_r$ and the union, since the union is infinite.
    However, we can split the indices in $\mathcal{I}$ into two families.
    If $\rho(i) < r$, since $f_i(\Lambda) \cap F \neq \varnothing$, $f_i(\Lambda)\subset F^{(r\cdot\diam X)}$.
    Moreover, since we are in $\R^d$, there is a constant $C > 0$ so that $N_\delta(E^{(\delta)}) \leq C N_\delta(E)$ for any set $E\subset\R^d$ and $\delta > 0$.
    For the terms $\rho(i) > r$, at least heuristically, swapping the union and covering number is not too lossy.
    Therefore, we decompose
    \begin{align*}
        N_r(\Lambda) &\leq N_r\left(\bigcup_{\substack{i\in\mathcal{I}\\\rho(i) N r}}f_i(\Lambda)\right) + \sum_{\substack{i\in\mathcal{I}\\\rho(i) \geq r}} N_r(f_i(\Lambda))\\
                     & \leq N_r(F^{(r\cdot\diam X)}) + \sum_{\substack{i\in\mathcal{I}\\\rho(i) \geq r}} N_r(f_i(\Lambda))\\
                     & \leq C N_r(F) + \sum_{\substack{i\in\mathcal{I}\\\rho(i) \geq r}} N_{r \rho(i)^{-1}}(\Lambda).
    \end{align*}
    In the last line, we used self-similarity: the set $f_i(\Lambda)$ is a self-similar copy of $\Lambda$ so $N_r(f_i(\Lambda)) = N_{r \rho(i)^{-1}}(\Lambda)$.

    Now, we can apply the exact same argument to each set $N_{r \rho(i)^{-1}}(\Lambda)$:
    \begin{align*}
        N_r(\Lambda) &\leq C N_r(F) + \sum_{\substack{i\in\mathcal{I}\\\rho(i) \geq r}} \left(C N_{r\rho(i)^{-1}} + \sum_{\substack{j\in\mathcal{I}\\\rho(j) \geq r \rho(i)^{-1}}}N_{r \rho(ij)^{-1}}\right)\\
                     &= CN_r(F) + \sum_{\substack{i\in\mathcal{I}\\\rho(i) \geq r}}C N_{r\rho(i)^{-1}}(F) + \sum_{\substack{\mtt{i}\in\mathcal{I}^2\\\rho(\mtt{i}) \geq r}} N_{r\rho(\mtt{i})^{-1}}(\Lambda).
    \end{align*}
    Now, recall that the IFS is uniformly contracting, so $\rho(\mtt{i}) \leq \xi^k$ for $\mtt{i}\in\mathcal{I}^k$.
    In other words, if we iterate the above procedure $k$ times where $k$ is sufficiently large so that $\xi^k < r$, the final sum becomes empty and we obtain the general upper bound
    \begin{equation}\label{e:cov}
        N_r(\Lambda) \leq C \sum_{\substack{\mtt{i}\in\mathcal{I}^*\\\rho(\mtt{i}) \geq r}}N_{r \rho(\mtt{i})^{-1}}(F).
    \end{equation}
    So, we have replaced the question of covering $\Lambda$ with a problem of bounding the covering numbers of many copies of $F$ at many different scales.
    In order to bound the above sum, we find a new representation of the sum in a more convenient form.
    Define
    \begin{equation}\label{e:theta-def}
        \theta_{\mtt{i}}(r) = 1 - \frac{\log \rho(\mtt{i})}{\log r}.
    \end{equation}
    Equivalently, $\theta_{\mtt{i}}(r)$ is chosen so that $r^{\theta_{\mtt{i}}(r)} = r \rho(\mtt{i})^{-1}$.
    Note that $\theta_{\mtt{i}}(r) \in [0,1]$, and therefore
    \begin{equation}\label{e:lossy}
        \begin{aligned}
            N_{r\rho(\mtt{i})^{-1}}(F) &\lesssim_\varepsilon \left(\frac{\rho(\mtt{i})}{r}\right)^{\dimuB F + \varepsilon}\\
                                       &= \rho(\mtt{i})^{h + \varepsilon}\left(\frac{1}{r}\right)^{\theta_{\mtt{i}}(r) (\dimuB F + \varepsilon) + (1-\theta_{\mtt{i}}(r))(h + \varepsilon)}\\
                                       &\leq \rho(\mtt{i})^{h + \varepsilon} \left(\frac{1}{r}\right)^{\max \{h, \dimuB F\} + \varepsilon}.
        \end{aligned}
    \end{equation}
    Substituting this back into \cref{e:cov} and then using \cref{e:total-bound},
    \begin{align*}
        N_r(\Lambda) &\lesssim_\varepsilon\sum_{\substack{\mtt{i}\in\mathcal{I}^*\\\rho(\mtt{i}) \geq r}}\rho(\mtt{i})^{h + \varepsilon}\left(\frac{1}{r}\right)^{\max \{h, \dimuB F\} + \varepsilon}\\
                     &\leq\left(\frac{1}{r}\right)^{\max \{h, \dimuB F\} + \varepsilon}\sum_{\mtt{i}\in\mathcal{I}^*}\rho(\mtt{i})^{h+\varepsilon}\\
                     & \lesssim_\varepsilon\left(\frac{1}{r}\right)^{\max \{h, \dimuB F\} + \varepsilon}.
    \end{align*}
    Therefore $\dimuB\Lambda \leq \max\{h, \dimuB F\} + \varepsilon$.
\end{proof}

\section{Lower box dimension and asymptotics for covering numbers}\label{s:scaling}
We now turn our attention to the lower box dimension.
Unlike the upper box dimension, it turns out that the case of the lower box dimension is quite a bit more subtle.

Let us begin with a simple bound, which only uses our results on the Hausdorff and upper box dimensions:
\begin{equation}\label{e:triv}
    \max\{h, \dimlB F\} \leq \dimlB\Lambda \leq \dimuB\Lambda = \max\{h, \dimuB F\}.
\end{equation}
In particular, this immediately provides two trivial mechanisms to guarantee that the box dimension exists:
\begin{enumerate}[nl,r]
    \item We have $h \geq \dimuB F$, in which case $\dimB\Lambda = h$.
        In other words, the Hausdorff dimension is sufficiently large that we can simply ignore the set of fixed points.
    \item We have $\dimlB F = \dimuB F$, in which case $\dimB\Lambda = \max\{h, \dimB F\}$.
        This case occurs when the fixed point set itself is sufficiently regular.
\end{enumerate}
Perhaps surprisingly, it turns out that these are the only two mechanisms under which the box dimension can exist.
In fact, we will prove the following result: $\dimlB \Lambda = \dimuB \Lambda$ if and only if $\max\{h, \dimlB F\} = \max\{h, \dimuB F\}$.
This was first proven in \cite{arxiv:2406.12821}.

Unfortunately, there is a problem: it turns out that there is no simple formula for $\dimlB\Lambda$.
More precisely, there is no formula for $\dimlB\Lambda$ which depends only on $h$, $\dimlB F$, and $\dimuB F$.

Let us briefly try to understand the difficult by inspecting the proof of \cref{t:dimuB-ss}, and trying to substitute the lower box dimension in place of the upper box dimension.
Heuristically speaking, \cref{e:cov} is not so lossy: below, we will give a different shorter proof of a slight generalization of \cref{e:cov} below in \cref{l:cov-alt}, which is clearly optimal.
The main loss occurs in the first inequality in \cref{e:lossy} if $\dimlB F < \dimuB F$.
Moreover, the optimal exponent here depends on the value of $\theta_{\mtt{i}}(r)$, which could take (essentially) any value in the interval $[0,1]$.
So, to understand the covering number at scale $r$, we need to know something about the covering numbers of $F$ at all scales between $r$ and $1$.

This motivates us to instead try to answer the following more general question: can we obtain an asymptotic formula for the covering numbers $r\mapsto N_r(\Lambda)$ given the Hausdorff dimension $h$ and the covering numbers $r\mapsto N_r(F)$?
We are asking for quite a bit more information about $\Lambda$, but in exchange we receive quite a bit more information about $F$.
It will turn out that this question has a concrete answer, and we can use the concrete answer to this question to provide sharp bounds on the lower box dimension (in terms of $h$, $\dimlB F$, and $\dimuB F$) which in turn provides the classification of existence of the box dimension of $\Lambda$.

\subsection{Branching functions}\label{ss:br}
Let us first briefly introduce a formalism around the covering numbers $r\mapsto N_r(F)$.
We begin with a normalization which is very convenient.
\begin{definition}
    Given a set $E\subset\R^d$, the \emph{branching function} of $E$ is the function $\beta_E\colon [0,\infty)\to [0,\infty)$ given by
    \begin{equation*}
        \beta_E(u) = \log N_{2^{-u}}(E).
    \end{equation*}
\end{definition}
Here, the logarithm is in base $2$.
The reason for choosing this normalization of the covering number is that the function $\beta_E$ is essentially an increasing $d$-Lipschitz function with $\beta_E(0) = 0$.
The $d$-Lipschitz property occurs for the following reason.
Given $u,v \geq 0$, we need to upper bound $N_{2^{-(u+v)}}(E)$ in terms of $N_{2^{-u}}(E)$.
First fix an optimal cover $\{B(x_1, 2^{-u}), \ldots, B(x_n, 2^{-u})\}$ for $E$, where $n = N_{2^{-u}}(E)$.
Then, working in $\R^d$, we can cover each ball $B(x_j, 2^{-u})$ with $C_d 2^{dv}$ balls of radius $2^{-(u+v)}$, where $C_d$ is some constant depending only on $d$.
Therefore
\begin{equation}\label{e:approx-lip}
    \beta_E(u+v) = \log N_{2^{-(u+v)}}(E) \leq \log \left(C_d 2^{dv} N_{2^{-u}}(E)\right) \leq \beta_E(u) + dv + \log C_d.
\end{equation}
In other words, $\beta_E$ is $d$-Lipschitz up to a fixed additive error term.

Here is a formalization which makes it slightly easier to work with the error terms.
Let $\mathcal{B}(d)$ denote the increasing $d$-Lipschitz functions $g\colon[0,\infty)\to[0,\infty)$ with $g(0) = 0$.
\begin{proposition}\label{p:approx-br}
    For all $d\in\N$, there is a constant $M_d \in\R$ such that the following holds:

    Let $E\subset\R^d$.
    Then there exists $g\in\mathcal{B}(d)$ so that for all $u\geq 0$,
    \begin{equation*}
        |\beta_E(u) - g(u)| \leq M_d + d\log(1 + \diam E).
    \end{equation*}
\end{proposition}
\begin{proof}
    Let $C_d > 0$ be chosen so that
    \begin{equation}\label{e:gb}
        N_{2^{-u}}(E) \leq C_d(\diam E)^d 2^{ud}
    \end{equation}
    and \cref{e:approx-lip} holds.
    Let $M_d = \log C_d$.
    Now, set
    \begin{equation*}
        g = \sup\{h\in\mathcal{B}(d): h \leq \beta_E\}.
    \end{equation*}
    The supremum is non-empty since $\beta_E \geq 0$.
    Moreover, it is easy to check that $\mathcal{B}(d)$ is closed under suprema, so $g\in\mathcal{B}(d)$.

    Now, let $u \geq 0$ be fixed.
    If $\beta_E(u) \leq M_d + d\log(1 + \diam E)$, there is nothing to prove.
    Otherwise, consider the piecewise linear function $\psi_u$ such that:
    \begin{enumerate}[nl,a]
        \item $\psi_u(u) = \beta_E(u) - M_d - d\log(1+\diam E) \geq 0$,
        \item $\psi_u$ has constant slope $d$ on $[0, u]$, and
        \item $\psi_u$ is constant on $[u,\infty)$.
    \end{enumerate}
    By \cref{e:approx-lip} and since $\beta_E$ is increasing, it follows that $\psi_u \leq \beta_E$ on $[0,\infty)$.
    Moreover, recall from \cref{e:gb}
    \begin{equation*}
        \beta_E(u) \leq M_d + d\log(\diam E) + ud.
    \end{equation*}
    In particular, $\psi_u(0) \leq 0$.
    Thus $\max\{0, \psi_u\} \in \mathcal{B}(d)$ so
    \begin{equation*}
        g(u) \geq \psi_u(u) = \beta_E(u) - M_d - d\log(1+\diam E).
    \end{equation*}
    Since $u\geq 0$ was arbitrary, the proof is complete.
\end{proof}
We can also choose branching functions which play nicely with the lower and upper box dimensions.
\begin{proposition}\label{p:dimb-bf}
    Let $E\subset\R^d$ and let $s = \dimlB E$ and $t = \dimuB E$.
    Then there exists $g\in\mathcal{B}(d)$ such that $g = \beta_E + o(u)$ and $s u \leq g(u) \leq t u$ for all $u \geq 0$.
\end{proposition}
\begin{proof}
    First apply \cref{p:approx-br} to get $f\in\mathcal{B}(d)$ such that $f = \beta_E + O(1)$, and then set
    \begin{equation*}
        g(u) =\min\{\max\{f(u), su\}, tu\}.
    \end{equation*}
    In other words, $g(u)$ is obtained by clamping $f(u)$ to the desired range $[su, tu]$
    Clearly $g\in\mathcal{B}(d)$; we must show that $g = f + o(u)$.

    Indeed, for any $\varepsilon>0$, for all $u$ sufficiently large by definition of the lower and upper box dimension, $(s-\varepsilon) u \leq f(u) \leq (t+\varepsilon) u$.
    Therefore either $f(u) \in [su, tu]$ in which case $f = g$, or $|f(u) - g(u)|\leq\varepsilon u$.
    Since this holds for all $\varepsilon>0$, it follows that $g = f + o(u)$.
\end{proof}

\subsection{A candidate asymptotic formula}
Recall our goal in this section: given an IFS of similarities $\{f_i\}_{i\in\mathcal{I}}$ satisfying the bounded neighbourhood condition, we wish to give essentially sharp bounds on $\beta_\Lambda(u)$ in terms of $h$ and $\beta_F(u)$.
More precisely, we will provide a formula for the equivalence class $\beta_\Lambda(u) + o(u)$ in terms of $h$ and $\beta_F(u) + o(u)$.
The $o(u)$ error term is acceptable since, for example,
\begin{equation*}
    \dimlB\Lambda = \liminf_{u\to\infty} \frac{\beta_\Lambda(u)}{u}.
\end{equation*}
\begin{remark}
    It is an interesting question to try to determine sharper asymptotics, but as we will see this would likely require (at the very least) additional assumptions on the similarity ratios $\rho(i)$.
    For an explicit formulation, see \cref{q:better-asymp}.
\end{remark}
Our strategy is essentially as follows.
For the upper bound, we will follow the proof of \cref{t:dimuB-ss} but whenever possible input information about $\beta_F$.
Then, we prove a matching lower bound by considering finite subsystems.

In the proof of \cref{t:dimuB-ss}, we first obtained the formula \cref{e:cov} using an inductive argument, and then used \cref{e:cov} along with some algebraic manipulation to obtain the upper bound on the upper box dimension.
It is not too difficult to furnish a direct proof of \cref{e:cov} with less effort.
Given $r \in(0,1)$, define
\begin{equation*}
    F^*(r) = \bigcup_{\substack{\mtt{i}\in\mathcal{I}^*\\\rho(\mtt{i}) \geq r}} f_{\mtt{i}}(F).
\end{equation*}
Of course, $F^*(r)\subset\Lambda$ for all $r\in(0,1)$, and $F^*(r) = F$ for $r$ sufficiently close to $1$.
Essentially as proven in \cref{e:cov}, the sets $F^*(r)$ provide good approximations to $\Lambda$ at scale $r$.
We recall that $E^{(\delta)}$ denotes open $\delta$-neighbourhood of $E$.
\begin{lemma}\label{l:cov-alt}
    Let $\{f_i\}_{i\in\mathcal{I}}$ be an IFS of similarities with attractor $\Lambda$ and fixed point set $F$.
    Then for all $r \in(0,1)$,
    \begin{equation*}
        \Lambda \subset (F^*(r))^{(r\cdot\diam X)}.
    \end{equation*}
\end{lemma}
\begin{proof}
    Let $x\in\Lambda$ be arbitrary.
    Iterating the self-similarity relation, there is an $\mtt{i}\in\mathcal{I}^*$ with $\rho(\mtt{i}) < r$, $\rho(\mtt{i}^-) \geq r$, and $x\in f_{\mtt{i}}(\Lambda)$.
    Write $\mtt{i} = \mtt{i}^- j$ and let $x_j\in F$ be the unique fixed point of $f_j$.
    Then
    \begin{equation*}
        f_{\mtt{i}^{-}}(x_j) = f_{\mtt{i}^-}(f_j(x_j)) = f_{\mtt{i}}(x_j)\in f_{\mtt{i}}(\Lambda).
    \end{equation*}
    Since $\rho(\mtt{i}^-) \geq r$, $f_{\mtt{i}^-}(x_j) \in F^*(r)$ by definition, so $f_{\mtt{i}}(\Lambda) \cap F^*(r) \neq \varnothing$.
    Since $\rho(\mtt{i}) < r$, $\diam f_{\mtt{i}}(\Lambda) < r \diam X$ so that
    \begin{equation*}
        x\in f_{\mtt{i}}(\Lambda) \subset (F^*(r))^{(r\cdot\diam X)}.
    \end{equation*}
    Since $x\in\Lambda$ was arbitrary, the proof is complete.
\end{proof}
Now, let us try to guess a formula for $\beta_\Lambda$, given $\beta_F$ and $h$.
Using \cref{l:cov-alt} and the fact that the $f_{\mtt{i}}$ are similarities,
\begin{equation*}
    N_r(\Lambda) \lesssim \sum_{\substack{\mtt{i}\in\mathcal{I}^*\\\rho(\mtt{i}) \geq r}} N_{r\rho(\mtt{i})^{-1}}(F).
\end{equation*}
The bounded neighbourhood condition implies this inequality is quite close to being an equality.
Now recall the definition of $\theta_{\mtt{i}} = \theta_{\mtt{i}}(r)$ from \cref{e:theta-def}.
Write $r = 2^{-u}$ for $u \geq 0$, so $r\rho(\mtt{i})^{-1} = 2^{-\theta_{\mtt{i}} u}$.
Taking logarithms and recalling the computation from the proof of \cref{t:dimuB-ss},
\begin{align*}
    \beta_\Lambda(u) &\leq \log\sum_{\substack{\mtt{i}\in\mathcal{I}^*\\\rho(\mtt{i}) \geq r}} N_{r\rho(\mtt{i})^{-1}}(F)= \log\sum_{\substack{\mtt{i}\in\mathcal{I}^*\\\rho(\mtt{i}) \geq r}} \rho(\mtt{i})^h 2^{(1-\theta_{\mtt{i}})\cdot h \cdot u + \beta_F(\theta_{\mtt{i}} u)}\\
                     &\leq \sup_{0 \leq z \leq u}\bigl(\beta_F(z) + h(u-z)\bigr) + \log\sum_{\mtt{i}\in\mathcal{I}^*}\rho(\mtt{i})^h
\end{align*}
In the second line, we used the fact that $\theta_{\mtt{i}} u$ takes values on $[0, u]$.
Swapping the summation and the supremum is not so lossy since, by pigeonholing and losing a $\log u$ error term, we may assume that $r\rho(\mtt{i})^{-1}$ is essentially constant, which does not cause issues since $\beta_F$ is essentially continuous.
This gives us the following candidate formula for $\beta_\Lambda$:
\begin{equation*}
    \beta_\Lambda(u) = \sup_{0 \leq z \leq u}\bigl(\beta_F(z) + h(u-z)\bigr) + o(u).
\end{equation*}
In order to make this sketch rigorous, we first prove a few essential technical facts.

\subsection{Technical results concerning level sets of cylinders}
The first technical concerns the summation defining the pressure function $P$.
\begin{lemma}\label{l:trunc-sum}
    For all $\varepsilon > 0$, there is a constant $C = C(\varepsilon) \geq 1$ so that
    \begin{equation*}
        \left(\frac{1}{r}\right)^{- \varepsilon}
        \leq
        \sum_{\substack{\mtt{i}\in\mathcal{I}^*\\Cr \geq \rho(\mtt{i}) \geq r}}\rho(\mtt{i})^h
        \leq
        \sum_{\substack{\mtt{i}\in\mathcal{I}^*\\\rho(\mtt{i}) \geq r}}\rho(\mtt{i})^h \lesssim_\varepsilon \left(\frac{1}{r}\right)^{\varepsilon}.
    \end{equation*}
\end{lemma}
\begin{proof}
    We begin with the upper bound.
    Recalling \cref{e:total-bound},
    \begin{equation*}
        \sum_{\substack{\mtt{i}\in\mathcal{I}^*\\\rho(\mtt{i}) \geq r}}\rho(\mtt{i})^h \leq\left(\frac{1}{r}\right)^{\varepsilon}\sum_{\mtt{i}\in\mathcal{I}^*}\rho(\mtt{i})^{h +\varepsilon} \lesssim_\varepsilon\left(\frac{1}{r}\right)^{\varepsilon}.
    \end{equation*}
    For the lower bound, we approximate the summation using finite subsets of $\mathcal{I}$.
    Since $\sum_{i\in\mathcal{I}}\rho(i)^{h-\varepsilon} > 1$, there exists $\mathcal{F}\in\Fin(\mathcal{I})$ and $t > h-\varepsilon$ so that $\sum_{i\in\mathcal{F}}\rho(i)^{t} = 1$.
    Let $C = \max\{\rho(i)^{-1} : i\in\mathcal{F}\}$, so every infinite word $x\in\mathcal{F}^{\N}$ contains at least one prefix with $r \leq \rho(\mtt{i}) \leq Cr$, and therefore
    \begin{equation*}
        r^{-\varepsilon}\sum_{\substack{\mtt{i}\in\mathcal{I}^*\\Cr \geq \rho(\mtt{i}) \geq r}} \rho(\mtt{i})^h
        \geq \sum_{\substack{\mtt{i}\in\mathcal{I}^*\\Cr \geq \rho(\mtt{i}) \geq r}} \rho(\mtt{i})^{h-\varepsilon}
        \geq \sum_{\substack{\mtt{i}\in\mathcal{F}^*\\Cr \geq \rho(\mtt{i}) \geq r}}\rho(\mtt{i})^t \geq 1.
    \end{equation*}
    The proof is complete after rearranging.
\end{proof}
\cref{l:trunc-sum} shows at the cost of a $(1/r)^\epsilon$ factor, we may restrict attention to maps with similarity ratio approximately $r$.
Using the bounded neighbourhood condition, we can transfer the symbolic result concerning the pressure to a geometric fact about the distribution of the cylinders $f_{\mtt{i}}(X)$.
\begin{lemma}\label{l:geom-red}
    Let $\{f_i\}_{i\in\mathcal{I}}$ be an IFS of similarities satisfying the bounded neighbourhood condition with constant $M$.
    Then for all $C \geq 1$ and $0<r \leq R < 1$,
    \begin{equation*}
        \frac{1}{M C^{d+1}}\sum_{\substack{\mtt{i}\in\mathcal{I}^*\\CR \geq \rho(\mtt{i}) \geq R}} N_{r/R}(F)
        \lesssim
        N_r\left(\bigcup_{\substack{\mtt{i}\in\mathcal{I}^*\\CR \geq \rho(\mtt{i}) \geq R}} f_{\mtt{i}}(F)\right)
        \lesssim
        C^d\mkern-10mu
        \sum_{\substack{\mtt{i}\in\mathcal{I}^*\\CR \geq \rho(\mtt{i}) \geq R}} N_{r/R}(F)
    \end{equation*}
\end{lemma}
\begin{proof}
    For the upper bound, note that $r\rho(\mtt{i})^{-1} \geq r/CR$, so the statement holds with $N_{r/CR}$ in place of $N_{r/R}$, but these covering numbers are the same up to a factor $C^d$ times a constant in $\R^d$.

    For the lower bound, we find a large subset of $\mathcal{M}\coloneqq\{\mtt{i}\in\mathcal{I}^*:R \leq\rho(\mtt{i}) \leq CR\}$ such that the sets $f_{\mtt{i}}(F)$ are $R$-separated.
    First, since each infinite $x\in\mathcal{I}^{\N}$ contains at most $\lesssim \log(C) + 1\lesssim C$ prefixes $\mtt{i}$ in $\mathcal{M}$, we may find a mutually incomparable sub-family $\mathcal{M}_1\subset\mathcal{M}$ such that $\#\mathcal{M}_1 \gtrsim \frac{1}{C}\#\mathcal{M}$.
    Next, since $\diam f_{\mtt{i}}(F) \lesssim CR$, we can cover $f_{\mtt{i}}(F)$ with $\lesssim C^d$ balls of radius $R$.
    For each such ball $B(x,R)$, by the bounded neighbourhood condition, there are at most $M$ words $\mtt{i}\in\mathcal{M}_1$ such that $f_{\mtt{i}}(F) \cap B(x,R) \neq \varnothing$.
    In particular, there are $\lesssim M C^d$ words $\mtt{j} \neq \mtt{i}$ such that $d(f_{\mtt{i}}(F), f_{\mtt{j}}(F)) \leq R$.
    Thus the greedy algorithm yields a subset $\mathcal{M}_2\subset\mathcal{M}_1$ with $\#\mathcal{M}_2 \gtrsim M^{-1}C^{-d}\#\mathcal{M}_1$ such that $d(f_{\mtt{i}}(F), f_{\mtt{j}}(F)) > R$ for all $\mtt{i} \neq \mtt{j}$ in $\mathcal{M}_2$.
    Since $r \leq R$, the lower bound follows.
\end{proof}

\subsection{Proof of the asymptotic formula}
We now have our main result on the covering numbers $r\mapsto N_r(\Lambda)$.
This result was first established in \cite{arxiv:2406.12821}, in the more general case that the maps are conformal.
\begin{theorem}\label{t:asymp}
    Let $\{f_i\}_{i\in\mathcal{I}}$ be an IFS of similarities satisfying the bounded neighbourhood condition with attractor $\Lambda$ and fixed point set $F$.
    Then its branching function satisfies
    \begin{equation*}
        \beta_\Lambda(u) = \sup_{0 \leq z \leq u}\bigl(\beta_F(z) + h(u-z)\bigr) + o(u).
    \end{equation*}
\end{theorem}
\begin{proof}
    Let $\varepsilon > 0$ be arbitrary and let $C = C_\varepsilon\geq 1$ be the corresponding constant guaranteed by \cref{l:trunc-sum}.
    Let $u \geq 0$ and set $r = 2^{-u}$ and let $r \leq R \leq 1$ be arbitrary.
    Let $\theta$ be chosen so that $r/R = 2^{-\theta u}$ and write $N_{r/R}(F) = R^h 2^{(1-\theta)h u + \beta_F(\theta u)}$.
    Since $R \leq \rho(\mtt{i}) \leq C R$, by \cref{l:geom-red},
    \begin{align*}
        N_r\left(\bigcup_{\substack{\mtt{i}\in\mathcal{I}^*\\CR \geq \rho(\mtt{i}) \geq R}}f_{\mtt{i}}(F)\right)
        &\approx_\varepsilon \sum_{\substack{\mtt{i}\in\mathcal{I}^*\\CR \geq \rho(\mtt{i}) \geq R}}N_{r/R}(F)\\
        &= 2^{(1-\theta)h u + \beta_F(\theta u)}\sum_{\substack{\mtt{i}\in\mathcal{I}^*\\CR \geq \rho(\mtt{i}) \geq R}}R^h\\
        &\approx_\varepsilon 2^{(1-\theta)h u + \beta_F(\theta u)}\sum_{\substack{\mtt{i}\in\mathcal{I}^*\\CR \geq \rho(\mtt{i}) \geq R}}\rho(\mtt{i})^h.
    \end{align*}
    By \cref{l:cov-alt} and the pigeonhole principle,
    \begin{equation*}
        \frac{\log C_\varepsilon}{1 + \log(1/r)}\cdot N_r(\Lambda) \leq \sup_{r \leq R \leq 1} N_r\left(\bigcup_{\substack{\mtt{i}\in\mathcal{I}^*\\CR \geq \rho(\mtt{i}) \geq R}}f_{\mtt{i}}(F)\right) \leq N_r(\Lambda).
    \end{equation*}
    As $R$ ranges from $r$ to $1$, $\theta$ ranges from $1$ to $0$.
    Thus applying \cref{l:trunc-sum}, there is a constant $M_\varepsilon \geq 0$ so that, taking logarithms,
    \begin{equation*}
        \left\lvert\beta_\Lambda(u) - \sup_{0 \leq z \leq u}\bigl(\beta_F(z) + h(u-z)\bigr)\right\rvert \leq \varepsilon u + M_\varepsilon.
    \end{equation*}
    Since $\varepsilon > 0$ was arbitrary, we are done.
\end{proof}

\subsection{The asymptotic formula as an operator}
Let $\mathcal{A}$ denote the family of increasing functions from $[0,\infty) \to [0,\infty)$, and recall the space $\mathcal{B}(\alpha)$ of functions which are $\alpha$-Lipschitz and satisfy $f(0) = 0$.
Certainly $\mathcal{B}(\alpha)\subset\mathcal{A}$; we recall from \cref{ss:br} that $\mathcal{B}(d)$ is essentially the space of branching functions of subsets of $\R^d$.

With \cref{t:asymp} in mind, for a number $0 \leq h$ and $f\in\mathcal{A}$, let
\begin{equation*}
    \Psi_h(f)(u) = \sup_{0 \leq z \leq u}\bigl(f(z) + h(u-z)\bigr).
\end{equation*}
Since $f$ is increasing, $\Psi_0$ is the identity map.
On the other hand, $\Psi_h$ for $h > 0$ is already a non-trivial operator; for instance $\Psi_h(f)(u) \geq h u$ for any $f\in\mathcal{A}$.

We begin by establish some basic properties of $\Psi_h$ as an operator.
\begin{proposition}\label{p:psi}
    Let $h \geq 0$ and $f,g\in\mathcal{A}$.
    The following hold:
    \begin{enumerate}[nl,r]
        \item\label{i:a} $\Psi_hf \in \mathcal{A}$.
        \item\label{i:b} If $h\leq \alpha$ and $f\in\mathcal{B}(\alpha)$, then $\Psi_hf \in\mathcal{B}(\alpha)$.
        \item\label{i:pres} If $\eta\in\mathcal{A}$ and $f \leq g + \eta$, then $\Psi_hf \leq \Psi_hg + \eta$.
        \item\label{i:op} If $f \leq g$, then $\Psi_hf \leq \Psi_hg$.
        \item\label{i:ou} If $f = g + o(u)$, then $\Psi_hf = \Psi_hg + o(u)$.
        \item\label{i:o1} If $f = g + O(1)$, then $\Psi_hf = \Psi_hg + O(1)$.
    \end{enumerate}
\end{proposition}
\begin{proof}
    For $f\in\mathcal{A}$ and $z \geq 0$, define
    \begin{equation*}
        f_z(u) = \begin{cases}
            f(u) &: 0 \leq u \leq z\\
            f(z) + h(u-z) &: z \leq u
        \end{cases}
    \end{equation*}
    Observe that $\Psi_hf = \sup_{z \geq 0}f_z$.
    If $f\in\mathcal{A}$, then $f_z\in\mathcal{A}$, and a supremum of increasing functions is increasing, so $\Psi_hf\in\mathcal{A}$ giving \cref{i:a}.
    Similarly, if $f\in\mathcal{B}(\alpha)$ where $\alpha\geq h$, then $f_z\in\mathcal{B}(\alpha)$; since $\mathcal{B}(\alpha)$ is closed under suprema, $\Psi_hf\in\mathcal{B}(\alpha)$ giving \cref{i:b}.

    Now, for \cref{i:pres}, observe that
    \begin{align*}
        \Psi_hf(u)
        &= \sup_{0\leq z \leq u}\bigl(f(z) + h(u-z)\bigr)\\
        &\leq \sup_{0\leq z \leq u}\bigl(g(z) + \eta(z) + h(u-z)\bigr)\\
        & \leq \Psi_hg(u) + \eta(u)
    \end{align*}
    since $\eta$ is increasing.
    Then \cref{i:op} follows by taking $\eta(u) = 0$, \cref{i:ou} follows by taking $\eta$ with $\lim_{u\to\infty}u^{-1}\eta(u) = 0$, and \cref{i:o1} follows by taking $\eta(u) = C$.
\end{proof}
Recall from \cref{p:approx-br} that if $E\subset\R^d$, then there is an $f\in\mathcal{B}(d)$ such that $\beta_E = f + O(1)$.
By \cref{p:psi}, the $O(1)$ error term is preserved by the operator $\Psi_h$, and $\Psi_h$ descends to an operator on $\mathcal{B}(d)$.
Therefore instead of thinking of $\Psi_h$ as acting on the branching function $\beta_F$, we think of $\Psi_h$ as acting on a representative $f\in\mathcal{B}(d)$ for $\beta_F$.

A convenient consequence of working with a representative $f\in\mathcal{B}(d)$ is that, by continuity of $f$, for each $u \geq 0$ there is a maximal $0 \leq u \leq \zeta_h(u)$ which realizes the supremum:
\begin{equation*}
    \Psi_h(f)(u) = f(\zeta_h(u)) + h(u-\zeta_h(u)).
\end{equation*}

Now we give a simple description of the operator $\Psi_h$ restricted to $\mathcal{B}(\alpha)$.
For $0 \leq h \leq \alpha$, set
\begin{equation*}
    \mathcal{B}_h(\alpha) = \{f\in\mathcal{B}(\alpha) : f(u+z) \geq f(u) + hz\text{ for all }0\leq z \leq u\}.
\end{equation*}
The space $\mathcal{B}_h(\alpha)$ is the set of branching functions which are \emph{lower $h$-Lipschitz}: if $\partial f(u)$ denotes the set of derivatives of $f$ at $u$, then $\mathcal{B}_h(\alpha)$ is the set of functions $f\colon[0,\infty)\to[0,\infty)$ with $f(0) = 0$ and $\partial f(u) \subset [h,\alpha]$ for all $u \geq 0$.
\begin{proposition}\label{p:proj-char}
    Let $0 \leq h \leq \alpha$ and $f\in\mathcal{B}(\alpha)$.
    Then
    \begin{equation*}
        \Psi_h(f) = \inf\{g\in\mathcal{B}_h(\alpha): g \geq f\}.
    \end{equation*}
    In particular, $\Psi_h\colon\mathcal{B}(\alpha)\to\mathcal{B}_h(\alpha)$ is surjective and idempotent.
\end{proposition}
\begin{proof}
    Let $f\in\mathcal{B}(\alpha)$.
    Clearly $\Psi_h f \geq f$.
    We next show that $\Psi_h f \in \mathcal{B}_h(\alpha)$.
    Recall that $\Psi_h f\in\mathcal{B}(\alpha)$ from \cref{p:psi}~\cref{i:b}, so it suffices to check the $h$-Lipschitz property.
    Let $0 \leq v \leq u$: then
    \begin{align*}
        \Psi_h f(u) &\geq h(u - \zeta_h(v)) + f(\zeta_h(v))\\
                    &= h(u-\zeta_h(v)) + \Psi_hf(v) - h(v - \zeta_h(v))\\
                    &= \Psi_h f(v) + h(u-v)
    \end{align*}
    as required.
    To complete the proof, suppose $g\in\mathcal{B}_h(\alpha)$ and $g \geq f$: we must show that $g \geq \Psi_hf$.
    Indeed, using the $h$-Lipschitz property of $g$,
    \begin{align*}
        \Psi_hf(u) = f(\zeta_h(u)) + h (u - \zeta_h(u))\leq g(\zeta_h(u)) + h (u - \zeta_h(u))\leq g(u).
    \end{align*}
    as required.
\end{proof}
The formula \cref{p:proj-char} has a natural geometric interpretation.
Momentarily forget about the formula for $\beta_\Lambda$, and assume we are given the values $h$ and the function $\beta_F$.
What can we say about $\beta_\Lambda$?
\begin{enumerate}[nl]
    \item Since $F\subset\Lambda$, we must have $\beta_F \leq \beta_\Lambda$.
    \item\label{i:lower} Suppose we know the value of $\beta_\Lambda(u)$.
        What can we say about $\beta_\Lambda(u+v)$?
        Let $\{B(x_i, 2^{-u})\}$ be a maximal family of disjoint balls with $x_i\in\Lambda$.
        Since $\Lambda$ is self-conformal, in a heuristic sense, we can find a copy of $\Lambda$ scaled by a factor of $2^{-u}$ inside each ball $B(x_i, 2^{-u})$.
        Since $\Lambda$ is $h$-dimensional, this means that $N_{2^{-(u+v)}}(B(x_i, 2^{-u})\cap \Lambda) \geq 2^{h(u+v)}$: in words, $\Lambda$ is ``at least $h$-dimensional in all balls''.
        But this property of being ``at least $h$-dimensional between all pairs of scales'', at the level of branching functions, is the same as saying that $\beta_\Lambda$ is lower $h$-Lipschitz.
\end{enumerate}
\cref{p:proj-char} says that these are the only two obstructions: $\beta_\Lambda$ is \emph{as small as possible}, given that it is bounded below by $\beta_F$ and lower $h$-Lipschitz.
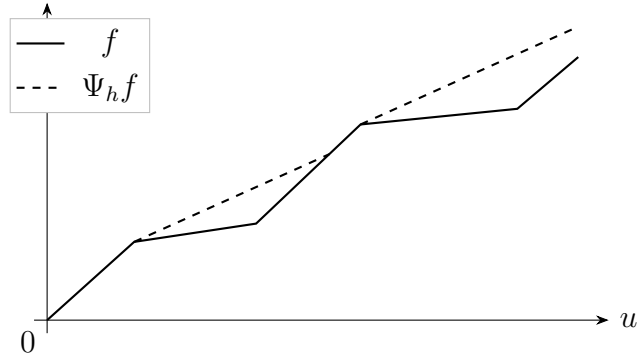
\begin{figure}[t]
    \centering
    \input{figures/bf_transform.tex}
    \caption{A branching function $f$ and the corresponding function $\Psi_hf$.}
    \label{f:bf-transform}
\end{figure}

Note, in the infinite case, that property \cref{i:lower} is not literally true: the issue is that in a ball $B(x,r)$ with $x\in\Lambda$, there need not exist a copy of $\Lambda$ rescaled by a factor of exactly $r$: the rescaled copy could be much smaller.
However, this is still true on average in a weak sub-exponential sense, and this is what is formalized in the proof of \cref{t:asymp}.

\subsection{Consequences for the lower box dimension}
Now, we use \cref{t:asymp} to obtain some applications concerning the lower box dimension.
The first application is a non-trivial bound on $\dimlB\Lambda$ in terms of $h$, $\dimlB F$, $\dimuB F$, and the ambient dimension $d$.
Given $0 \leq s\leq t \leq \alpha$ and $0\leq h\leq \alpha$, let
\begin{equation*}
    \mathcal{D}(h,s,t,\alpha) = \begin{cases}
        \{h\} &: t \leq h,\\
        \left[\max\{h,s\}, h + \frac{(t-h)(\alpha-h) s}{\alpha t - hs}\right] &: t > h.
    \end{cases}
\end{equation*}
Note that $\mathcal{D}(h,s,t,\alpha)$ is not a singleton if and only if $0<h<t$ and $0<s<t$.
\begin{corollary}\label{c:dimlb}
    Let $\{f_i\}_{i\in\mathcal{I}}$ be an IFS of similarities satisfying the bounded neighbourhood condition with attractor $\Lambda$ and fixed point set $F$.
    Then
    \begin{equation*}
        \dimlB\Lambda \in\mathcal{D}(\dimH\Lambda,\dimlB F,\dimuB F, d).
    \end{equation*}
\end{corollary}
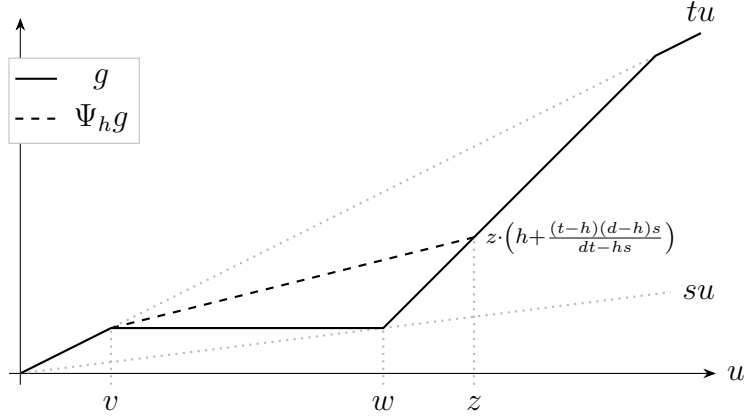
\begin{figure}[t]
    \centering
    \input{figures/dimlb_bound.tex}
    \caption{
        The upper bound $g$ in the proof of \cref{c:dimlb}, in the special case that $f(w)=sw$.
        The dashed line has slope $h$.
        The function $f$ (not depicted) is some branching function bounded above by $g$ and below by $su$.
    }
    \label{f:dimlb-bound}
\end{figure}
\begin{proof}
    Write $s = \dimlB F$ and $t = \dimuB F$.
    If $t \leq h$ or $s = t$, we already saw that $\dimlB\Lambda = \max\{h,s\}$, and in these cases $\mathcal{D}(h,s,t,d) = \{\max\{h, s\}\}$.
    If $\dimuB F \leq h$ or $\dimlB F = \dimuB F$, we already saw that $\dimlB \Lambda = h$.
    Therefore we may assume that $s < t$ and $h < t$.
    By \cref{p:dimb-bf}, get $f\in\mathcal{B}(d)$ such that $f = \beta_F + o(u)$, $f(u) \in [su, tu]$ for all $u \geq 0$.
    Then $\beta_\Lambda = \Psi_h f + o(u)$ by \cref{t:asymp} and \cref{p:psi}~\cref{i:ou}.

    Let $s < s' < t$ be arbitrary.
    By definition of $s = \dimlB F$, there exists arbitrarily large $w\geq 0$ such that $f(w) \leq s' w$.
    Fix such a value $w$, let $v \leq w$ be such that $t v = s' w$ and consider the function
    \begin{equation*}
        g(u) = \begin{cases}
            t u &: u \leq v,\\
            tv &: v\leq u \leq w,\\
            tv + d(u-w) &: w \leq u.
        \end{cases}
    \end{equation*}
    Clearly $g\in\mathcal{B}(d)$.
    Moreover, since $f(u) \leq tu$ and $f(w) \leq s' w$, $f \leq g$.
    In particular, by \cref{p:psi}~\cref{i:op}, $\Psi_h f \leq \Psi_h g$.

    Since $h < t \leq d$, $g$ is bounded above by the affine line $\ell(u) = tv + h (u-v)$ on the interval $[0, z]$ where $z$ is the largest intersection point of $g$ and $\ell$ (and the only intersection point other than $(v, \ell(v))$).
    For this value $z$,
    \begin{equation*}
        \Psi_h f(z) \leq \Psi_h g(z) = \ell(z).
    \end{equation*}
    Writing $v$ and $z$ in terms of $w$ (and the ambient data $s'$, $t$, $h$, and $d$), a direct computation shows that
    \begin{equation*}
        \frac{\ell(z)}{z} = h + \frac{(t-h)(d-h) s'}{d t - hs'}.
    \end{equation*}
    Since this holds for arbitrarily large $z$, and since $s' > s$ was arbitrary, we conclude that $\dimlB\Lambda \in \mathcal{D}(h,s,t,d)$ as claimed.
\end{proof}
\begin{remark}
    In fact, the bounds in \cref{c:dimlb} are sharp in the following sense: given any values for $0 < h<d$, $0\leq s \leq t \leq d$, and $\beta\in\mathcal{D}(h,s,t,d)$ there exists an IFS of similarities on $\R^d$ satisfying the bounded neighbourhood condition with attractor $\Lambda$ and fixed point set $F$ such that $\dimH\Lambda = h$, $\dimlB F = s$, $\dimuB F = t$, and $\dimlB\Lambda = \beta$.
    See \cref{ss:sharp} for more details.
\end{remark}
\begin{remark}\label{r:inf}
    The sets $\mathcal{D}(h,s,t,\alpha)$ are monotonically increasing in $\alpha$.
    Moreover,
    \begin{align*}
        \mathcal{D}(h,s,t,\infty)\coloneqq{}&\lim_{\alpha\to\infty}\mathcal{D}(h,s,t,\alpha)\\
        ={}&
        \begin{cases}
            \{h\} &: t \leq h,\\
            \left[\max\{h,s\}, s + \bigl(1-\frac{s}{t}\bigr)h\right] &: t > h.
        \end{cases}
    \end{align*}
    This gives non-trivial bounds which are valid in all dimensions simultaneously.
\end{remark}
Using \cref{c:dimlb}, we now obtain our characterization of the existence of the box dimension of $\Lambda$.
This fact is intuitively clear from \cref{f:dimlb-bound}: as long as $s < t$ and $h < t$, the dashed line intersects the function $g$ strictly below the line with slope $t$.
\begin{corollary}\label{c:dimlb-char}
    Let $\{f_i\}_{i\in\mathcal{I}}$ be an IFS of similarities satisfying the bounded neighbourhood condition with attractor $\Lambda$ and fixed point set $F$.
    Then $\dimlB\Lambda = \dimuB\Lambda$ if and only if $\dimuB F \leq \max\{h, \dimlB F\}$.
\end{corollary}
\begin{proof}
    Recalling the earlier discussion in \cref{e:triv}, if $\dimuB F \leq \max\{h, \dimlB F\}$, then it follows immediately that $\dimlB\Lambda = \dimuB\Lambda$.

    Conversely, suppose that $\dimuB F > \max\{h, \dimlB F\}$.
    In particular,
    \begin{equation*}
        \dimuB\Lambda > h \qquad\text{and}\qquad 1 - \frac{\dimuB F}{\dimlB F} > 0.
    \end{equation*}
    Thus by \cref{c:dimlb} and \cref{r:inf},
    \begin{align*}
        \dimlB\Lambda &\leq \dimlB F + \left(1-\frac{\dimuB F}{\dimlB F}\right)h\\
                      &<\dimlB F + \left(1-\frac{\dimuB F}{\dimlB F}\right)\dimuB F\\
                      &=\dimuB F\\
                      &\leq \dimuB \Lambda
    \end{align*}
    as claimed.
\end{proof}

\section{Various loose ends}\label{s:final}
We now tie up a variety of loose ends, sometimes with full proofs and sometimes with less detail.
Most notably, we discuss the general self-conformal case, discuss a few other notions of dimension, and conclude with some open problems.

\subsection{Infinitely-generated self-conformal sets}\label{ss:conf}
Let us first discuss what must be changed to handle more general self-conformal iterated function systems.
For us, the following definition will be sufficient.
\begin{definition}\label{d:conf}
    We say that the IFS $\{f_i\}_{i\in\mathcal{I}}$ is \emph{weakly conformal} if there are constants $c > 0$, $K \geq 1$, $\xi\in(0,1)$, and a function $\rho\colon\mathcal{I}^* \to (0,1]$ such that:
    \begin{enumerate}[nl,r]
        \item $\rho(\varnothing) = 1$,
        \item\label{i:ctr} $\rho(i) \leq \xi$ for all $i\in\mathcal{I}$,
        \item\label{i:submul} $K^{-1}\rho(\mtt{i})\rho(\mtt{j}) \leq \rho(\mtt{i}\mtt{j}) \leq \rho(\mtt{i})\rho(\mtt{j})$, and
        \item for all $r > 0$ sufficiently small, $x\in X$, and $r > 0$,
            \begin{equation}\label{e:rho-ball-distort}
                B\bigl(f_{\mtt{i}}(x), c\rho(\mtt{i}) r\bigr) \subseteq f_{\mtt{i}}(B(x,r)) \subseteq B\bigl(f_{\mtt{i}}(x), \rho(\mtt{i}) r\bigr).
            \end{equation}
    \end{enumerate}
\end{definition}
\begin{remark}
    Condition \cref{i:ctr} is just strict contractivity of the IFS.

    Since the ball $B(x,r)$ may not be contained in $X$, \cref{d:conf} implicitly assumes that that there is an open neighbourhood $V$ of $X$ such that each $f_{\mtt{i}}$ may be extended to map on $V$.
\end{remark}
\begin{example}
    Let us verify that these assumptions are satisfied for the inverse branches of the Gauss map.
    For each function $\phi_b$, recall that $\phi_b(x) = 1/(b + x)$.
    Let us prove, by induction on $n$, that for each $\mtt{b}\in\N^n$,
    \begin{equation}\label{e:deriv-ind}
        \phi_{\mtt{b}}'(x) = \frac{(-1)^n}{(px + q)^2}
    \end{equation}
    for integers $1 \leq p \leq q$.
    (In fact, the constants $p$ and $q$ are precisely the denominators of the convergents at steps $n-1$ and $n$ respectively.)
    Clearly this is true for $n = 1$.
    Now let $n \geq 1$ be arbitrary and suppose $\mtt{b}\in\N^{n+1}$.
    Write $\mtt{b} = \mtt{b}^- j$ so
    \begin{equation*}
        \phi_{\mtt{b}}'(x) = \frac{(-1)^n}{(p(j + x)^{-1} + q)^2}\cdot\frac{-1}{(j+x)^2}= \frac{(-1)^{n+1}}{\bigl(qx + (q j + p)\bigr)^2}
    \end{equation*}
    which is of the claimed form.

    Now, it is easy to check using \cref{e:deriv-ind} that $|\phi_{\mtt{b}}'(x)| \leq 4|\phi_{\mtt{b}}'(y)|$ for all $x,y\in [0,1]$; and therefore in some small open neighbourhood $[0,1]\subset V$ we have that $|\phi_{\mtt{b}}'(x)| \leq 5|\phi_{\mtt{b}}'(y)|$ for all $x,y\in V$.
    Thus we can define the function $\rho(\mtt{b}) = \sup_{x\in V}|\phi_{\mtt{b}}'(x)|$ and the IFS is weakly conformal with constants $c^{-1} = K = 5$.

\end{example}
Similarly to the definition of the pressure in the self-similar case, the function $\rho$ allows us to define the pressure in the self-conformal case:
\begin{equation*}
    P(t) = \lim_{m\to\infty}\frac{1}{m}\log S_m(t)\qquad\text{where}\qquad S_m(t)\coloneqq\sum_{\mtt{i}\in\mathcal{I}^m}\rho(\mtt{i})^t.
\end{equation*}
By \cref{d:conf}~\cref{i:submul}, it always holds that $K^{-t} S_k(t)S_m(t) \leq S_{k+m}(t) \leq S_k(t)S_m(t)$.
Therefore, the limit defining $P(t)$ always exists and moreover that $S_m(t) < \infty$ if and only if $S_1(t) < \infty$.
Moreover, since the IFS is strictly contracting, we can write
\begin{equation*}
    \frac{1}{m}\log\sum_{\mtt{i}\in\mathcal{I}^m}\rho(\mtt{i})^t = t \log\xi +\frac{1}{m}\log \sum_{\mtt{i}\in\mathcal{I}^m}(\xi^{-m}\rho(\mtt{i}))^t.
\end{equation*}
Since $\xi^{-m}\rho(\mtt{i}) \leq 1$, it follows that $P(t) - t\log\xi$ is decreasing, so in particular $P(t)$ is strictly decreasing and $\lim_{t\to\infty}P(t) = -\infty$ as long as $P(t) < \infty$ for some $t$.
With this in mind, if $\eta = \inf\{t \geq 0:S_1(t) < \infty\}$ then $P(t) < \infty$ for all $t >\eta$.
In particular,
\begin{equation}\label{e:fin-m}
    \sum_{\mtt{i}\in\mathcal{I}^m}\rho(\mtt{i})^{\eta+\varepsilon} < \infty.
\end{equation}
If $\rho$ is multiplicative, then
\begin{equation*}
    \sum_{\mtt{i}\in\mathcal{I}^m}\rho(\mtt{i})^t=\left(\sum_{i\in\mathcal{I}}\rho(i)^t\right)^m
\end{equation*}
so the definition of the pressure coincides with our earlier definition in the special case that the $f_{\mtt{i}}$ are similarities.
In order to work with inequalities involving the pressure, we must therefore work with the higher iterates $\{f_{\mtt{i}}\}_{\mtt{i}\in\mathcal{I}^n}$.

For the upper bound, we no longer (in general) control sums of the form $\sum_{\mtt{i}\in\mathcal{I}^*}\rho(\mtt{i})^h$.
Instead, we can obtain bounds of the form
\begin{equation*}
    \sum_{\substack{\mtt{i}\in(\mathcal{I}^m)^*\\\rho(\mtt{i}) > r}}\rho(\mtt{i})^h \lesssim_\varepsilon\left(\frac{1}{r}\right)^{\varepsilon}
\end{equation*}
where $m$ must be chosen sufficiently large in terms of $\varepsilon$.
Since the ``step size'' is now of order $m$, we instead obtain bounds in terms of the set $F_m$ of fixed points of $\{f_{\mtt{i}}:\mtt{i}\in\mathcal{I}^m\}$.
Note that $\beta_{F_m}$ may be much larger than $\beta_{F_1}$: regardless, we still expect that $\beta_{F_m}$ is bounded above by $\Psi_h\beta_{F_1}$ (for instance, this must be true if the main result holds).
Moreover, since $\Psi_h$ is idempotent, this means that $\Psi_h\beta_{F_1} = \Psi_h\beta_{F_m}$.
So that the argument is not circular, we must show up-front that that $\Psi_h\beta_{F_m} = \Psi_h\beta_{F_1}$.
This can be done directly with an induction argument.

It turns out that instead of working with sets of fixed points, it is more convenient to work with the level $m$ orbit set
\begin{equation*}
    F_m\coloneqq \bigcup_{\mtt{i}\in\mathcal{I}^m}f_{\mtt{i}}(x_0)
\end{equation*}
where $x_0$ is chosen so that $f_i(x_0) = x_0$ for some $i\in\mathcal{I}$.
It is not too difficult to show that if $F$ is the set of fixed points and $F_1$ is the orbit set above, then $\beta_F = \beta_{F_1} + O(1)$; see \cref{p:approx-orbit}.
With this choice of $x_0$, observe that
\begin{equation}\label{e:F-invariance}
    F_{m+1}=\bigcup_{\mtt{i}\in\mathcal{I}^{m}}f_{\mtt{i}}(F_1) \subset \Lambda.
\end{equation}
As before, for each $\mtt{i}\in\mathcal{I}^*$ with $\rho(\mtt{i})>r = 2^{-u}$, let $\theta_{\mtt{i}}(r)\in(0,1]$ be chosen so that $r \rho(\mtt{i})^{-1} = r^{\theta_\mtt{i}(r)}$ and therefore
\begin{equation}\label{e:lambda-choice}
    N_{r\cdot \rho(\mtt{i})^{-1}}(E) =  \rho(\mtt{i})^h\cdot 2^{ (1-\theta_{\mtt{i}}(r)) h u + \theta_{\mtt{i}}(r) \beta_E(u)}
\end{equation}
for a general set $E$.
Also, recall that
\begin{equation}\label{e:C-choice}
    N_r\left(\bigcup_{\substack{\mtt{i}\in\mathcal{I}^m\\\rho(\mtt{i})\leq r}}f_{\mtt{i}}(\Lambda)\right) \lesssim N_r(F_m).
\end{equation}
since each $f_{\mtt{i}}(\Lambda)$ where $\rho(\mtt{i}) \leq r$ is contained in the $r\cdot(\diam X)$-neighbourhood of $F_m$.
We now establish invariance under higher iterates.
The idea here is similar to the idea in the proof of \cite[Lemma~2.10]{zbl:0940.28009}.
\begin{lemma}\label{l:higher}
    Let $\{f_i\}_{i\in\mathcal{I}}$ be a weakly conformal IFS.
    Then for each $m\in\N$,
    \begin{equation*}
        \Psi_h \beta_{F_m}(u) = \Psi_h\beta_{F_1}(u) + o_m(u).
    \end{equation*}
\end{lemma}
\begin{proof}
    Since $F_1\subset F_2\subset\cdots$, it follows that $\Psi_1\beta_{F_1}\leq \Psi_h\beta_{F_m}$ for all $m\in\N$.
    Thus it suffices to prove for all $\varepsilon>0$ and all $m\in\N$ that
    \begin{equation}\label{e:limsup-version}
        \limsup_{u\to \infty}\bigl(\Psi_h\beta_{F_m}(u)-\Psi_h\beta_{F_1}(u)\bigr)\leq 2\varepsilon u.
    \end{equation}

    Fix $\varepsilon>0$.
    We first prove by induction on $m$ that there are constants $A_m > 0$ so that for all $r = 2^{-u} > 0$ sufficiently small,
    \begin{equation}\label{e:sn-bound}
        N_r(F_m) \leq A_m 2^{\Psi_h\beta_{F_1}(u) + u\varepsilon}.
    \end{equation}
    First, recall from \cref{e:fin-m} that
    \begin{equation}\label{e:pressure-sum}
        B_m \coloneqq \sum_{\mtt{i}\in\mathcal{I}^m} \rho(\mtt{i})^{\eta +\varepsilon} < \infty.
    \end{equation}
    Now we proceed with the induction.
    The case $m=1$ is trivial since $\beta_{F_1} \leq \Psi_h\beta_{F_1}$.
    Now suppose we have established \cref{e:sn-bound} for some $m\in\N$.
    In the computation below the implicit constant depends only on the ambient dimension $d$.
    For each $r = 2^{-u} > 0$ sufficiently small to apply \cref{e:rho-ball-distort}, recalling \cref{e:F-invariance} and \cref{e:C-choice},
    \begin{align*}
        N_r(F_{m+1}) &= N_r\left(\bigcup_{\mtt{i}\in\mathcal{I}^{m}} f_{\mtt{i}}(F_1)\right)\\
                 &\lesssim N_r(F_m) + \sum_{\substack{\mtt{i}\in\mathcal{I}^{m}\\\rho(\mtt{i})>r}}N_{r\cdot \rho(\mtt{i})^{-1}}(F_1)\\
                 &= N_r(F_m) + \sum_{\substack{\mtt{i}\in\mathcal{I}^{m}\\\rho(\mtt{i})>r}}\rho(\mtt{i})^h\cdot 2^{ (1-\theta_{\mtt{i}}(r)) h u + \theta_{\mtt{i}}(r) \beta_{F_1}(u)}\\
                 &\leq N_r(F_m) + 2^{\Psi_h\beta_{F_1}(u)+u\varepsilon}\sum_{\mtt{i}\in\mathcal{I}^{m}} \rho(\mtt{i})^{h+\varepsilon}\\
                 &\leq (A_m + B_m)\left(\frac{1}{r}\right)^{\Psi_h\beta_{F_1}(r)+\varepsilon}.
    \end{align*}
    In the last line, we have used the inductive hypothesis and \cref{e:pressure-sum}.
    Thus \cref{e:sn-bound} follows.
    We have shown that for all $m\in\N$ that there is a constant $a_m \in\R$ so that
    \begin{equation*}
        \beta_{F_m}(u) \leq\Psi_h\beta_{F_1}(u) + 2\varepsilon u + a_m.
    \end{equation*}
    Since $\varepsilon>0$ was arbitrary, it follows that $\beta_{F_m}(u) \leq \Psi_h\beta_{F_1}(u) + o_m(u)$.
    Applying the operator $\Psi_h$ on both sides of this inequality and using \cref{p:psi}~\cref{i:pres} and idempotence of $\Psi_h$ from \cref{p:proj-char}, the proof is complete.
\end{proof}
\begin{remark}
    We actually proved something somewhat stronger: if $\{f_i\}_{i\in\mathcal{I}}$ is an IFS for which there is a function $\rho\colon\mathcal{I}^* \to (0,1]$ such that $\rho(\varnothing) = 1$, $\rho(\mtt{i}\mtt{j}) \leq \rho(\mtt{i})\rho(\mtt{j})$, and
    \begin{equation*}
        f_{\mtt{i}}(B(x,r)) \subseteq B\bigl(f_{\mtt{i}}(x), \rho(\mtt{i}) r\bigr),
    \end{equation*}
    then
    \begin{equation*}
        \Psi_\eta \beta_{F_m}(u) = \Psi_\eta\beta_{F_1}(u) + o_m(u)
    \end{equation*}
    where $\eta = \inf\{t \geq 0:\sum_{i\in\mathcal{I}}\rho(i)^t < \infty\}$ is the usual finiteness parameter.
\end{remark}

\begin{example}
    Using \cref{c:dimlb-char}, or more precisely, the version of \cref{c:dimlb-char} for self-conformal IFSs which we have outlined in this section, we can give an explicit example of a set $I\subset\N$ for which the lower and upper box dimensions do not coincide.

    The first thing to keep in mind is that we require $h < 1/2$: since $\dimB F = \dimB \{n^{-1}:n\in\N\}=1/2$, if $h \geq 1/2$, then $\dimH\Lambda_{\D} = \dimB\Lambda_{\D}$.
    Since the finiteness parameter is $1/2$, in order to lower the Hausdorff dimension we must first lower the finiteness parameter.

    So, we first consider the digit set $\mathcal{D}\coloneqq\{n^{p}:n\in\N, n\geq N\}$ for some $p > 1$ and large $N\in\N$.
    This set has finiteness parameter $1/2p$ and therefore, taking $N$ to be large, the Hausdorff dimension of $\Lambda_{\mathcal{D}}$ can be made arbitrarily close to $1/2p$.

    On the other hand, one can compute that
    \begin{equation*}
        \dimB F = \frac{p}{p+1} > \frac{1}{2p}\qquad\text{where}\qquad F \coloneqq \{n^{-p}: n\in\N, n\geq N\},
    \end{equation*}
    independently of the choice of $N$.
    Moreover, the set of fixed points is essentially $F$.

    So, to complete the construction, by \cref{c:dimlb-char}, we just need to modify $\mathcal{D}$ so that the lower box dimension is strictly less than $p/(p+1)$.
    This is quite easy; simply delete long stretches of consecutive entries but sufficiently sparsely so that the upper box dimension is unchanged.
\end{example}

\subsection{Alternatives to the fixed point set}\label{ss:approx}
In this section, we discuss a minor technical fact concerning the choice of the set $F$ of fixed points.
\begin{definition}
    Fix an IFS $\{f_i\}_{i\in\mathcal{I}}$ acting on a compact set $X$.
    We say that a set $E\subset\bigcup_{i\in\mathcal{I}}f_i(X)$ is an \emph{approximate orbit} if there is an integer $\ell$ such that
    \begin{equation*}
        1 \leq \#(E\cap f_i(X)) \leq \ell
    \end{equation*}
    for all $i\in\mathcal{I}$.
\end{definition}
Assuming the bounded neighbourhood condition, it is easy to check that fixed point set $F$ and the orbit sets $\{f_i(x_0):i\in\mathcal{I}\}$ (for any $x_0\in X$) are approximate orbits.

From the perspective of branching functions, all approximate orbits are equivalent (with implicit constants depending on the compact set $X$, the integer $\ell$, and the constant implicit in the definition of the bounded neighbourhood condition).
This is essentially \cite[Proposition~2.9]{zbl:0940.28009}.
\begin{proposition}\label{p:approx-orbit}
    Let $A$ and $B$ be approximate orbits of an IFS $\{f_i\}_{i\in\mathcal{I}}$ satisfying the bounded neighbourhood condition.
    Then
    \begin{equation*}
        \beta_{A} = \beta_{B} + O(1).
    \end{equation*}
\end{proposition}
\begin{proof}
    Let $r \in (0, 1)$ and let $\mathcal{F} = \{i\in\mathcal{I}:\rho(i) > r \}$.
    We partition $A = A_{\leq r} \cup A_{> r}$ where $A_{> r} = A \cap\bigcup_{i\in\mathcal{F}}f_i(X)$, and similarly for $B$.

    If $i\notin\mathcal{F}$, then $\diam f_i(X) \leq \rho(i) \diam X \leq r \diam X$.
    Since $B\cap f_i(X) \neq\varnothing$, it follows that $A_{\leq r}$ is contained in the $r\cdot (\diam X)$-neighbourhood of $B$.
    Of course, the same also holds with $A$ and $B$ swapped, and therefore
    \begin{equation*}
        N_r(A_{\leq r}) \lesssim N_r(B)\qquad\text{and}\qquad N_r(B_{\leq r}) \lesssim N_r(A).
    \end{equation*}
    To complete the proof, it suffices to show that
    \begin{equation*}
        N_r(A_{> r}) \approx \#\mathcal{F}\approx N_r(B_{> r}).
    \end{equation*}
    Let $M$ be the constant implicit in the bounded neighbourhood condition, so that
    \begin{equation*}
        \#\{i\in\mathcal{F}:f_{i}(X)\cap B(x,r)\neq\varnothing\}\leq M\qquad\text{for all $x\in X$}.
    \end{equation*}
    Fix a cover $\{B(x_i,r)\}_{i=1}^n$ where $n = N_r(A_{>r})$.
    On one hand, $B(x_i, r)$ intersects at most $M$ images $f_i(X)$, so $\#\mathcal{F} \leq n M$.
    On the other hand, each $f_i(X)$ contains at most $\ell$ points in $A$, so $n \leq \ell \#\mathcal{F}$.
    Of course, the same holds with $B$ in place of $A$, as required.
\end{proof}

\subsection{Sharpness of the bounds on the lower box dimensions}\label{ss:sharp}

We now explain why the bounds in \cref{c:dimlb} are sharp.

We focus our attention on the case $d=1$ for notational simplicity; see \cite{arxiv:2406.12821} for the general case.
We will prove the following result:
\begin{proposition}\label{p:sharp}
    For any $0 < h\leq 1$, $0\leq s \leq t \leq 1$, and $\beta\in\mathcal{D}(h,s,t,1)$ there exists an IFS of similarities on $[0,1]$ satisfying the bounded neighbourhood condition with attractor $\Lambda$ and fixed point set $F$ such that $\dimH\Lambda = h$, $\dimlB F = s$, $\dimuB F = t$, and $\dimlB\Lambda = \beta$.
\end{proposition}
We will prove something a bit more general, and obtain \cref{p:sharp} by applying the result to a particularly chosen branching function $f$.
\begin{theorem}\label{t:exist}
    Let $f\in\mathcal{B}(1)$ and $h\in(0,1]$.
    Then there exists an IFS of similarities $\{f_i\}_{i\in\mathcal{I}}$ with fixed point set $F$ and attractor $\Lambda$ such that $\dimH\Lambda = h$ and $f(u) = \beta_F(u) + o(u)$.
    In particular,
    \begin{equation*}
        \beta_\Lambda(u) = \Psi_h f(u) + o(u).
    \end{equation*}
\end{theorem}
We prove this theorem in two steps:
\begin{enumerate}[nl]
    \item Given $f\in\mathcal{B}(1)$, we construct a compact set $K\subset [0,1]$ with measure 0 and at most one accumulation point and with branching function $\beta_{K}(u) = f(u) + o(u)$.
    \item Given a compact set $K\subset [0,1]$ of measure 0, we construct an IFS $\{f_i\}_{i\in\mathcal{I}}$ such that $K$ is approximately the fixed point set of the IFS.
\end{enumerate}
We begin by constructing a compact set corresponding to an arbitrary branching function.
\begin{lemma}\label{l:ctr-F}
    For all $f\in\mathcal{B}(1)$, there exists a non-empty compact set $K\subset [0,1]$ of measure 0 such that
    \begin{equation*}
        \beta_{K}(u) = f(u) + o(u).
    \end{equation*}
\end{lemma}
\begin{proof}
    We begin by constructing a compact set $K_f$ which is quite uniform in space, and then construct $F$ as a subset of $K_f$.

    First, by taking a supremum over all integer-valued $1$-Lipschitz functions $h$ with $h(0) = 0$ bounded above by $f$, there exists a function $h\colon\Z\cap[0,\infty)\to\Z\cap[0,\infty)$ which is also an increasing and $1$-Lipschitz function with $h(0) = 0$, such that
    \begin{equation*}
        f(k) - 1 < h(k) \leq f(k)\quad\text{for all}\quad k\geq 0.
    \end{equation*}
    Then, define a sequence $(a_k)_{k=1}^\infty \in \{0,1\}^{\N}$ by the rule $a_k = h(k) - h(k-1)$.

    We inductively construct nested families of dyadic interval as follows.
    Let $\mathcal{Q}_0 = \{[0,1]\}$.
    Now, suppose we have constructed $\mathcal{Q}_k$ as a union of dyadic intervals at level $k$, for some $k \geq 0$.
    Then, replace each interval $Q\in\mathcal{Q}_k$ with $2^{a_{k+1}}$ sub-intervals at level $k+1$.
    Finally, set
    \begin{equation*}
        K_f = \bigcap_{k=0}^\infty\bigcup_{Q\in\mathcal{Q}_k}Q.
    \end{equation*}
    Note that $\mathcal{Q}_k$ is a union of exactly $2^{h(k)}$ dyadic intervals, so $\beta_{K_f}(u) = f(u) + O(1)$.

    To guarantee that $K$ has measure 0, we modify the sequence $(a_k)_{k=1}^\infty$ so that $a_{n_k} = 0$ along a density 0 subsequence $n_k$.
    The resulting function is $f + o(u)$ and it is clear from the above construction that $K$ has measure 0.
\end{proof}
Now given the set $K$, we construct the IFS of similarities.
This is particularly easy in $\R$.
\begin{lemma}\label{l:IFS-exist}
    Let $K\subset [0,1]$ be any non-empty compact set of Lebesgue measure 0.
    Then there exists an IFS of similarities $\{f_i\}_{i\in\mathcal{I}}$ satisfying the bounded neighbourhood condition with attractor $\Lambda$ such that $\dimH\Lambda = h$ and with an approximate orbit $F$ such that $\overline{F}\cup\{0,1\} = K \cup\{0,1\}$.
\end{lemma}
\begin{proof}
    Enumerate the maximal disjoint open intervals of $K^c$ as $\{J_i\}_{i\in\mathcal{I}}$ for a countable index set $\mathcal{I}$.
    For each $i\in\mathcal{I}$, write $J_i = (a_i, b_i)$, and let $f_i\colon (0,1)\to J_i$ be defined by the rule
    \begin{equation*}
        f_{i}(x) = (b_{i} - a_i)^{1/h} x + a_i.
    \end{equation*}
    Since the intervals $J_i$ are pairwise disjoint, it is easy to see that the IFS satisfies the bounded neighbourhood condition.
    Moreover, by definition, $K\cup\{0, 1\} = \overline{\{f_i(0):i\in\mathcal{I}\}}\cup\{0,1\}$.

    The IFS $\{f_i\}_{i\in\mathcal{I}}$ has pressure function
    \begin{equation*}
        P(t) = \sum_{i\in\mathcal{I}} (b_i - a_i)^{t/h}.
    \end{equation*}
    Since $F$ has at most 1 accumulation point, $\sum_{i=1}^\infty (b_i - a_i) = 1$, so $P(h) = 1$, as required.
\end{proof}
\begin{proofref}{t:exist}
    Let $f\in\mathcal{B}(1)$, and by \cref{l:ctr-F}, get a countable set $E\subset[0,1$ with at most $1$ accumulation point such that $\beta_{E}(u) = f(u) + o(u)$.
    Then apply \cref{l:IFS-exist} with the set $E$ and value $h$ to get an IFS $\{f_i\}_{i\in\mathcal{I}}$ satisfying the bounded neighbourhood condition with attractor $\Lambda$ and fixed point set $F$, such that $\beta_F(u) = \beta_E(u) + o(u) = f(u) + o(u)$ and $\dimH\Lambda =h$.
    The proof is therefore complete by \cref{t:asymp}.
\end{proofref}
Now we can complete the proof of the sharpness result.
\begin{proofref}{p:sharp}
    If $\mathcal{D}(h,s,t,1)$ is a singleton, let $f\in\mathcal{B}(1)$ be such that $\liminf_{u\to\infty}u^{-1}f(u) = s$ and $\liminf_{u\to\infty}u^{-1}f(u) = t$.
    Apply \cref{t:exist} to get an IFS with fixed point set $F$ and attractor $\Lambda$ such that $\beta_F = f + o(u)$ and $\dimH\Lambda = h$.
    Then by \cref{c:dimlb}, since $\mathcal{D}(h,s,t,1)$ is a singleton, it must be that $\dimlB\Lambda$ is the expected value.

    Otherwise, $\mathcal{D}(h,s,t,1)$ is not a singleton, so $0 < h < t$ and $0 < s < t$.
    For simplicity, assume $\max\{h,s\}<\beta$ and $t < 1$; in the general case, apply the construction inductively with sequences $\beta \leq \beta_n \leq t_n \leq t$ such that $\beta_n > \max\{h, s\}$ decreases to $\beta$ and $t_n < 1$ increases to $1$.

    The building block is a function $\psi\colon[0,1]\to [0,t]$ which is increasing, $1$-Lipschitz, and satisfies $\psi(u) = tu$ for all $u$ in a neighbourhood of $0$ and $u = 1$.
    The function $\psi$ is defined in a similar way as depicted in \cref{f:dimlb-bound}: fix values $0 < v < w < 1$; on $[0,v]$, let $\psi$ have slope $t$, on $[v,w]$, let $\psi$ have slope $0$, and on $[w, 1]$ let $\psi$ have slope $\alpha$.
    Adjusting $v$, $w$, and $\alpha$ and using the fact that $\beta\in\mathcal{D}(h,s,t,1)$, it can be arranged so that $\psi(z)/z=\beta$ where $z$ is the unique value for which the segment joining $(v, tv)$ and $(z, \psi(z))$ has slope $h$.

    Finally, to construct the function $f$, set $f(u) = tu$ for $u \leq 1$, and for $v^{n-1} \leq u \leq v^{n}$ define $f(u) = v^n \psi(u/v^n)$.
    Now $f$ on $[v^{n-1},v^n]$ is precisely a self-similar copy of $\psi$ on $[v,1]$, and the choice of $\psi$ ensures, by \cref{t:asymp}, that the corresponding IFS provided by \cref{t:exist} with branching function $f$ and value $h$ has the desired dimensions.
\end{proofref}

\subsection{Intermediate dimensions}
The intermediate dimensions are a continuously parametrized family of dimensions which lie between the Hausdorff and box dimensions, first introduced in \cite{zbl:1448.28009}.
The motivation is the following: since the box dimensions are defined using covers consisting of a fixed radius, whereas the Hausdorff dimension is defined using arbitrary covers, one can obtain a notion of dimension by allowing some flexibility of the size of the covers.
More precisely, the covers are permitted to lie in a block of scales of the form $[r, r^\theta]$ where $0 < \theta \leq 1$.
\begin{definition}\label{d:cov-cost}
    Given $0<\theta \leq 1$ and $r > 0$, the \emph{$(s, r, \theta)$-covering cost} is the quantity
    \begin{equation*}
        \mathcal{C}_r^{s,\theta}(E) = \inf\Bigl\{\sum_i(\diam A_i)^s: E\subset\bigcup_i A_i\text{ and }\diam A_i \in [r,r^\theta]\text{ for all $i$}\Bigr\}.
    \end{equation*}
\end{definition}
Now, the \emph{upper intermediate dimensions} of $E$ are given by
\begin{align*}
    \overline{\dim}_{\theta} E = \inf\bigl\{s \geq 0: \mathcal{C}^{s,\theta}_r(E) \lesssim_s 1\text{ for all }0<r < 1\bigr\}.
\end{align*}
For example, $\mathcal{C}_r^{s,1}(E) = r^s N_r(E)$, so $\overline{\dim}_1 E = \dimuB E$.

Analogously to the results for the upper box dimensions, in \cite{zbl:1547.28014}, Banaji \& Fraser proved that for an infinitely generated self-conformal set satisfying the usual conditions, for all $\theta \in (0,1]$,
\begin{equation*}
    \overline{\dim}_\theta\Lambda = \max\{h, \overline{\dim}_\theta F\}.
\end{equation*}
Similarly to the proof due to Mauldin \& Urbański in the case of the upper box dimension \cite{zbl:0940.28009}, they use an induction argument.

We give a simplified proof of this result in the self-similar case using the approach in \cref{t:dimuB-ss}.
In the proof of \cref{t:dimuB-ss}, we needed upper bounds for the covering numbers $N_{r}(f_{\mtt{i}}(F))$ where $\rho(\mtt{i}) > r$.
By self-similarity, this number is precisely $N_{r^{\theta_\mtt{i}(r)}}(F)$, where we recall that $r^{\theta_{\mtt{i}}(r)} = r \rho(\mtt{i})^{-1}$.

For the upper intermediate dimensions, the bound is slightly different:
\begin{align*}
    \mathcal{C}_r^{s,\theta}&(f_{\mtt{i}}(F))\\
    &=\inf\Bigl\{\sum_i(\diam A_i)^s: f_{\mtt{i}}(F)\subset\bigcup_i A_i\text{ and }\diam A_i \in [r,r^\theta]\Bigr\}\\
    &=\inf\Bigl\{\sum_i(\rho(\mtt{i})\cdot\diam B_i)^s: F\subset\bigcup_i B_i\text{ and }\diam B_i \in [\rho(\mtt{i})^{-1}r, \rho(\mtt{i})^{-1}r^\theta]\Bigr\}.
\end{align*}
In particular, we now have two cases to consider: either $r \leq \rho(\mtt{i}) \leq r^\theta$, or $r^\theta < \rho(\mtt{i}) < 1$.
Equivalently, these are the cases $0 \leq \theta_{\mtt{i}}(r) \leq 1-\theta$ and $1 -\theta  < \theta_{\mtt{i}}(r) < 1$.
The former case is easy, since $\rho(\mtt{i})$ is a permitted diameter so we can cover each $f_{\mtt{i}}(F)$ with a bounded number of balls.
In the latter case, we can use the upper intermediate dimensions with a smaller value of $\theta$: set
\begin{equation*}
    \kappa_{\mtt{i}}(r, \theta) = \frac{\theta_{\mtt{i}}(r) + \theta - 1}{\theta_{\mtt{i}}(r)} \in (0,\theta]
\end{equation*}
so that
\begin{equation}\label{e:rescale}
    \mathcal{C}_r^{s,\theta}(f_{\mtt{i}}(F)) = \rho(\mtt{i})^s\cdot\mathcal{C}_{r^{\theta_{\mtt{i}}(r)}}^{s, \kappa_{\mtt{i}}(r, \theta)}(F)\leq \rho(\mtt{i})^s\cdot\mathcal{C}_{r^{\theta_{\mtt{i}}(r)}}^{s, \theta}(F).
\end{equation}
Here is the formal version of the above argument.
\begin{theorem}\label{t:int}
    Let $\{f_i\}_{i\in\mathcal{I}}$ be an IFS of similarities with attractor $\Lambda$ and fixed point set $F$.
    Then
    \begin{equation*}
        \overline{\dim}_\theta\Lambda \leq \max\{h, \overline{\dim}_\theta F\}.
    \end{equation*}
    In particular, if $\dimH\Lambda \geq h$ (for instance, if the IFS satisfies the bounded neighbourhood condition), then equality holds.
\end{theorem}
\begin{proof}
    Let $r\in(0,1)$ and $s > \max\{h, \overline{\dim}_\theta F\}$.
    By \cref{l:cov-alt},
    \begin{equation*}
        \Lambda \subset (F^*(r))^{(r\cdot\diam X)}
    \end{equation*}
    so that
    \begin{equation*}
        \mathcal{C}_r^{s,\theta}(\Lambda)  \lesssim \mathcal{C}_r^{s,\theta}\left(\sum_{\substack{\mtt{i}\in\mathcal{I}^*\\\rho(\mtt{i}) > r}}f_{\mtt{i}}(F)\right) \leq\mathcal{C}_r^{s,\theta}\left(\sum_{\substack{\mtt{i}\in\mathcal{I}^*\\r^\theta \geq \rho(\mtt{i}) > r}}f_{\mtt{i}}(F)\right)+\sum_{\substack{\mtt{i}\in\mathcal{I}^*\\\rho(\mtt{i}) > r^\theta}}\mathcal{C}_r^{s,\theta}(f_{\mtt{i}}(F)).
    \end{equation*}
    To bound the first term, since $\diam f_{\mtt{i}}(F) \lesssim r^\theta$ and $s > h$,
    \begin{equation*}
        \mathcal{C}_r^{s,\theta}\left(\sum_{\substack{\mtt{i}\in\mathcal{I}^*\\r^\theta \geq \rho(\mtt{i}) > r}}f_{\mtt{i}}(F)\right) \lesssim
        \sum_{\substack{\mtt{i}\in\mathcal{I}^*\\r^\theta \geq \rho(\mtt{i}) > r}}\rho(\mtt{i})^s
        \leq
        \sum_{\mtt{i}\in\mathcal{I}^*}\rho(\mtt{i})^s \lesssim_s 1.
    \end{equation*}
    To bound the second term, by \cref{e:rescale} and since in addition $s > \overline{\dim}_\theta F$,
    \begin{equation*}
        \sum_{\substack{\mtt{i}\in\mathcal{I}^*\\\rho(\mtt{i}) > r^\theta}}\mathcal{C}_r^{s,\theta}(f_{\mtt{i}}(F)) \leq
        \sum_{\substack{\mtt{i}\in\mathcal{I}^*\\\rho(\mtt{i}) > r^\theta}}\rho(\mtt{i})^s \mathcal{C}_{r^{\theta_{\mtt{i}}(r)}}^{s,\theta}(F) \lesssim_s \sum_{\mtt{i}\in\mathcal{I}^*}\rho(\mtt{i})^s \lesssim_s 1.
    \end{equation*}
    Therefore $\mathcal{C}_r^{s,\theta}(\Lambda) \lesssim_s 1$, and since $s > \max\{h, \overline{\dim}_\theta F\}$ was arbitrary this completes the proof.
\end{proof}
\begin{remark}
     Banaji \& Fraser also obtained some partial results for the Assouad dimension (and other related notions of dimension) in \cite{zbl:1558.28002}.
     See that paper for definitions and more discussion of the results.
\end{remark}

\subsection{Some open problems}
In \cref{t:asymp} we proved an asymptotic formula for the covering numbers $N_r(\Lambda)$ in terms of the Hausdorff dimension $h$ and $N_r(F)$.
However, the asymptotic formula was only up to error terms of the form $C_\varepsilon (1/r)^{\varepsilon}$, for arbitrarily small $\varepsilon$.
There are multiple places where this error term enters the proof, but one notable location where it seems rather unavoidable is when bounding sums of the form $\sum_{\mtt{i}\in\mathcal{I}^*}\rho(\mtt{i})^h$.
If such sums do not introduce error terms, is it possible to get better asymptotics?
In the best case scenario, it might be possible to get an error term which is an absolute constant independent of $r$.
\begin{question}\label{q:better-asymp}
    Let $\{f_i\}_{i\in\mathcal{I}}$ be an IFS of similarities satisfying the bounded neighbourhood condition with attractor $\Lambda$ and fixed point set $F$.
    Suppose $\{f_i\}_{i\in\mathcal{I}}$ is \emph{regular}, in that the value $h$ satisfies $\sum_{i\in\mathcal{I}}\rho(i)^h = 1$.
    Does its branching function satisfy
    \begin{equation*}
        \beta_\Lambda(u) = \sup_{0 \leq z \leq u}\bigl(\beta_F(z) + h(u-z)\bigr) + O(1)?
    \end{equation*}
    What about the more general self-conformal case?
\end{question}
See \cite{zbl:0852.28005} for more discussion concerning the assumption that the IFS is regular.

In \cref{t:int} we saw that the results concerning the upper box dimension extended quite easily to the case of the upper intermediate dimensions.
One can also define the \emph{lower intermediate dimensions}:
\begin{align*}
    \underline{\dim}_{\theta} E = \inf\bigl\{s \geq 0: \mathcal{C}^{s,\theta}_r(E) \lesssim_s 1\text{ for arbitrarily small }r > 0\bigr\}.
\end{align*}
As before, $\underline{\dim}_1 E = \dimlB E$.

In order to understand the lower box dimension, we required information about the covering numbers $N_r(F)$ at all scales $r$.
We can also consider such information for the intermediate dimensions.
For values $0 \leq v \leq u$, define
\begin{equation*}
    \mathcal{W}_E(s, u, v) = \inf\Bigl\{\sum_i(\diam A_i)^s: E\subset\bigcup_i A_i\text{ and }\diam A_i \in [2^{-u}, 2^{-(u-v)}]\Bigr\}.
\end{equation*}
This is essentially \cref{d:cov-cost} but with a somewhat different parametrization.
One can show that $s\mapsto \log\mathcal{W}_E(s, u, v)$ is strictly decreasing and continuous, and approximately satisfies $\log\mathcal{W}_E(0, u,v)\geq 0$ and $\log\mathcal{W}_E(d, u,v) \leq 0$ in $\R^d$.
Therefore, one may define (somewhat informally) a function $\omega_E(u,v)$ so that
\begin{equation*}
    \log\mathcal{W}_E((u-v)^{-1}\omega_E(u,v), u, v) = 0.
\end{equation*}
For example,
\begin{equation*}
    0 = \log\mathcal{W}_E(u^{-1}\omega_E(u,0), u, 0) = \beta_E(u) - \omega_E(u,0)
\end{equation*}
so $\omega_E(u,0) = \beta_E(u)$ recovers the usual branching function.
Moreover, one can check for $0 < \theta \leq 1$ that $\omega_E$ encodes the intermediate dimensions:
\begin{align*}
    \overline{\dim}_{\theta} E &= \limsup_{u\to\infty}\frac{\omega_E(u, (1-\theta)u)}{\theta u}\\
    \underline{\dim}_{\theta} E &= \liminf_{u\to\infty}\frac{\omega_E(u, (1-\theta)u)}{\theta u}.
\end{align*}
Here is an explicit question.
\begin{question}\label{q:lower-int}
    Let $\{f_i\}_{i\in\mathcal{I}}$ be an IFS of similarities satisfying the bounded neighbourhood condition with attractor $\Lambda$ and fixed point set $F$.
    What can be said about $\underline{\dim}_\theta \Lambda$?
    Can one obtain a similar classification as proven in \cref{c:dimlb-char}?
    More generally, can one provide a formula for $\omega_\Lambda$ in terms of $h$ and $\omega_F$?
\end{question}

\begin{acknowledgements}
    These notes were prepared for a minicourse taught during the semester program ``Continued Fractions, Fractals, Ergodic theory and Dynamics'' during the Simons semester at IMPAN in Warsaw.
    My stay was supported by a Simons Foundation grant (award no.\ SFI-MPS-T-Institutes-00010825) and from Polish State Treasury funds (agreement no.\ MNiSW/2025/DAP/491).
\end{acknowledgements}
\end{document}

%% file: figures/bf_transform.tex
\begin{tikzpicture}[
    x=1.15cm,
    y=1.15cm,
    >=Stealth,
    fgraph/.style={thick},
    psigraph/.style={thick,dashed}
]
    \draw[->] (-0.15,0) -- (6.45,0) node[right] {$u$};
    \draw[->] (0,-0.15) -- (0,3.65);

    \draw[psigraph] (1,0.9) -- (3.24,1.908);
    \draw[psigraph] (3.6,2.25) -- (6.1,3.375);

    \draw[fgraph]
        (0,0) -- (1,0.9) -- (2.4,1.11) -- (3.6,2.25) -- (5.4,2.43) -- (6.1,3.025);

    \node[anchor=north east] at (0,0) {$0$};

    \matrix [
        anchor=north west,
        fill=white,
        draw=gray!50,
        thin,
        row sep=2pt,
        column sep=4pt,
        inner sep=2pt,
        nodes={anchor=center}
    ] at ([yshift=-6pt]current bounding box.north west) {
        \draw[fgraph, black] (0,0) -- (0.55,0); & \node {$f$}; \\
        \draw[psigraph, black] (0,0) -- (0.55,0); & \node {$\Psi_h f$}; \\
    };
\end{tikzpicture}

%% file: figures/dimlb_bound.tex
\def\s{0.6}
\begin{tikzpicture}[
    x=1.0cm,
    y=1.0cm,
    >=Stealth,
    fgraph/.style={gray!55},
    ggraph/.style={thick},
    psigraph/.style={thick, dashed},
    guide/.style={thick, gray!55, dotted}
]
    \draw[->] (-0.15,0) -- (9.2,0) node[right] {$u$};
    \draw[->] (0,-0.15) -- (0,4.7);

    \draw[guide] (0,0) -- (8.6,1.075) node[black, right] {$su$}; %
    \draw[guide] (0,0) -- (9,4.5) node[black, above] {$tu$}; %

    \coordinate (start) at (2 * \s, \s);
    \coordinate (min) at (10 * \s, 3 * \s);

    \draw[ggraph] (0,0)
        -- (start)
        -- (8 * \s, \s)
        -- (14 * \s, 7 * \s)
        -- (9, 4.5);

    \draw[psigraph] (start) -- (min) node[right] {$\scriptstyle z\cdot \left(h + \frac{(t-h)(d-h)s}{dt - hs}\right)$};

    \draw[guide] (2*\s, -0.15) node[below, black]{$v$} -- (2*\s, \s);
    \draw[guide] (8*\s, -0.15) node[below, black]{$w$} -- (8*\s, \s);
    \draw[guide] (10*\s, -0.15) node[below, black]{$z$} -- (min);

    \matrix [
        anchor=north west,
        fill=white,
        draw=gray!50,
        thin,
        row sep=2pt,
        column sep=4pt,
        inner sep=2pt,
        nodes={anchor=center}
    ] at ([yshift=-25pt]current bounding box.north west) {
        \draw[ggraph, black] (0,0) -- (0.55,0); & \node {$g$}; \\
        \draw[psigraph, black] (0,0) -- (0.55,0); & \node {$\Psi_h g$}; \\
    };
\end{tikzpicture}